\documentclass[hidelinks,onefignum,onetabnum]{siamart250211}

\usepackage{lipsum}
\usepackage{amsfonts}
\usepackage{graphicx}
\usepackage{epstopdf}
\usepackage{algorithmic}
\usepackage{enumitem}
\usepackage{xcolor}
\usepackage{tikz}
\usepackage{float} 
\usetikzlibrary{arrows.meta, positioning, calc, shapes, fit, backgrounds}
\ifpdf
  \DeclareGraphicsExtensions{.eps,.pdf,.png,.jpg}
\else
  \DeclareGraphicsExtensions{.eps}
\fi

\newcommand{\E}{\mathbb{E}}   
\newcommand{\Var}{\mathrm{Var}}
\newcommand{\Cov}{\mathrm{Cov}}
\newcommand{\eps}{\varepsilon}

\newsiamremark{remark}{Remark}
\newsiamremark{hypothesis}{Hypothesis}
\crefname{hypothesis}{Hypothesis}{Hypotheses}
\newsiamthm{claim}{Claim}
\newsiamremark{fact}{Fact}
\crefname{fact}{Fact}{Facts}

\headers{SAPNN Neural Surrogates for Kinetic UQ}{W. Chen, G. Dimarco, and L. Pareschi}

\title{Structure and asymptotic preserving deep neural surrogates for uncertainty quantification in multiscale kinetic equations\thanks{The work of Wei Chen was partially supported by the China Scholarship Council, and the National Natural Science Foundation of China (NSFC) through the Research Fund for the Doctoral Program (No. 125B2021). Wei Chen also acknowledge the hospitality of the University of Ferrara. 
  The work of Giacomo Dimarco was partially supported by the Italian Ministry of University and Research (MUR) through the PRIN
2020 project (No. 2020JLWP23) ``Integrated Mathematical Approaches to Socio–Epidemiological Dynamics”. The work of Lorenzo Pareschi was partially supported by the Royal Society under the Wolfson Fellowship ``Uncertainty quantification, data-driven simulations and learning of multiscale complex systems governed by PDEs". The partial support by the European Union through the Future Artificial Intelligence Research (FAIR) Foundation, ``MATH4AI" Project and by ICSC -- Centro Nazionale di Ricerca in High Performance Computing, Big Data and Quantum Computing, funded by European Union -- NextGenerationEU and by the Italian Ministry of University and Research (MUR) through the PRIN 2022 project (No. 2022KKJP4X) ``Advanced numerical methods for time dependent parametric partial differential equations with applications" is also acknowledged. This work has been written within the activities of GNFM and GNCS groups of INdAM (Italian National Institute of High Mathematics).}}

\author{Wei Chen\thanks{School of Mathematical Sciences, Xiamen University, China 
  (\email{weichenmath@stu.xmu.edu.cn}).}
\and Giacomo Dimarco\thanks{Department of Mathematics and Computer Science, University of Ferrara, Italy 
  (\email{giacomo.dimarco@unife.it}).}
\and Lorenzo Pareschi\thanks{Maxwell Institute for Mathematical Sciences and Department of Mathematics, School of Mathematical and Computer Sciences, Heriot-Watt University, Edinburgh, UK
  (\email{l.pareschi@hw.ac.uk}). Also affiliated with Department of Mathematics and Computer Science, University of Ferrara, Italy.}}

\usepackage{amsopn}

\ifpdf
\hypersetup{
  pdftitle={Structure and asymptotic preserving deep neural surrogates for uncertainty quantification in multiscale kinetic equations},
  pdfauthor={W. Chen, G. Dimarco, and L. Pareschi}
}
\fi

\begin{document}

\maketitle

\begin{abstract}
The high dimensionality of kinetic equations with stochastic parameters poses major computational challenges for uncertainty quantification (UQ). Traditional Monte Carlo (MC) sampling methods, while widely used, suffer from slow convergence and high variance, which become increasingly severe as the dimensionality of the parameter space grows.
To accelerate MC sampling, we adopt a multiscale control variates strategy that leverages low-fidelity solutions from simplified kinetic models to reduce variance. To further improve sampling efficiency and preserve the underlying physics, we introduce surrogate models based on structure and asymptotic preserving neural networks (SAPNNs). These deep neural networks are specifically designed to satisfy key physical properties, including positivity, conservation laws, entropy dissipation, and asymptotic limits.
By training the SAPNNs on low-fidelity models and enriching them with selected high-fidelity samples from the full Boltzmann equation, our method achieves significant variance reduction while maintaining physical consistency and asymptotic accuracy.
The proposed methodology enables efficient large-scale prediction in kinetic UQ and is validated across both homogeneous and nonhomogeneous multiscale regimes. Numerical results demonstrate improved accuracy and computational efficiency compared to standard MC techniques.
\end{abstract}



\tableofcontents

\section{Introduction}

Kinetic equations naturally arise in the statistical description of large systems of particles evolving over time and are employed to model a wide range of phenomena across diverse fields, including rarefied gas dynamics, plasma physics, biology, and socioeconomics \cite{Cercignani1988Y,dimarco2014numerical,villani2002review}. 
The Boltzmann equation, as the most fundamental kinetic equation, describes the statistical behavior of a thermodynamic system by accounting for both particle transport and binary collisions \cite{Cercignani1988Y}:
\begin{equation}
\label{kinetic}
  \partial_t f + v \cdot \nabla_x f = \frac{1}{\eps} Q(f, f),
\end{equation}
with the nonnegative distribution function $f = f(z, x, v, t)$, velocity $v \in \mathbb{R}^{d_v}$, position $x \in \Omega_x \subset \mathbb{R}^{d_x}$, time $t \geq 0$, and a random variable $z \in \Omega_z \subset \mathbb{R}^{d_z}$, where $d_v, d_x, d_z \geq 1$. 
The parameter $\eps > 0$ is the Knudsen number, and the collision operator which describes the effects of particle interactions can be written as:
\begin{equation}
  Q(f, f)(z, v) \!= \!\int_{S^{d_v - 1} \times \mathbb{R}^{d_v}} \!B(v, v_*, \omega, z)(f(z, v') f(z, v_*') - f(z, v) f(z, v_*))  {\rm d}v_*{\rm d}\omega,
\label{eq:coll}
\end{equation}
where the dependence from $x$ and $t$ has been omitted and
\begin{equation*}
  v' = \frac12 (v + v_*) + \frac12 (\vert v - v_* \vert \omega), \quad v'_* = \frac12 (v + v_*) - \frac12 (\vert v - v_* \vert \omega).
\end{equation*}
Although the Boltzmann equation is widely applicable, the complexity of the collision operator $Q(f,f)$ presents significant challenges for both analytical treatment and numerical computation. 
To address this, simplified collision models have been developed to capture the key features of the full Boltzmann operator. 
Among these, the Bhatnagar-Gross-Krook (BGK) model \cite{bhatnagar1954model}, which posits a simple relaxation process toward equilibrium, has been widely adopted due to its computational efficiency.
For BGK model, the collision operator is defined by
\begin{equation}
  Q(f, f) = \mu(M[f] - f),
  \label{eq:BGK}
\end{equation}
where $\mu >0$ denotes the collision frequency and the local Maxwellian function has the form
\begin{equation*}
  M[f] = \frac{\rho}{(2 \pi R T)^{d_v / 2}} \exp{\left(- \frac{\vert u - v \vert^2}{2 R T}\right)},
\end{equation*}
with the macroscopic density $\rho$, mean velocity $u$ and temperature $T$
\begin{equation*}
  \rho = \int_{\mathbb{R}^{d_v}} f ~{\rm d}v, \quad u = \frac{1}{\rho} \int_{\mathbb{R}^{d_v}} vf ~{\rm d}v, \quad T = \frac{1}{d_v R \rho} \int_{\mathbb{R}^{d_v}} (v - u) ^ 2f ~{\rm d}v.
\end{equation*}
Here, $R$ is called the gas constant.
By multiplying the BGK model by its collision invariants $\phi(v) = \left( 1, v, |v|^2/2 \right)^{\rm T}$ and integrating over velocity space, we obtain the moment equations:
\begin{equation}
\label{moment}
  \partial_t \int_{\mathbb{R}^{d_v}} f \phi(v) ~{\rm d}v + \nabla_x \cdot \int_{\mathbb{R}^{d_v}} v f \phi(v) ~{\rm d}v = 0.
\end{equation}
This procedure ensures consistency with the macroscopic conservation laws of mass, momentum, and energy.

Moreover, both the Boltzmann and BGK models satisfy an entropy inequality that reflects the second law of thermodynamics. This is typically formulated through the entropy functional
\begin{equation}
  H(f) = \int_{\mathbb{R}^{d_v}} f \log f ~{\rm d}v,
  \label{eq:H}
\end{equation}
which is non-increasing in time for spatially homogeneous solutions. This property plays a crucial role in ensuring physically realistic dynamics and will be used later in our framework to calibrate the BGK model.

As $\eps \to 0$, the closed system of compressible Euler equations is formally derived through \eqref{moment} \cite{dimarco2014numerical}:
\begin{equation}
\label{Euler}
\left\{
\begin{aligned}
  &\partial_t \rho + \nabla_x \cdot \left( \rho u\right) = 0, \\
  &\partial_t \left(\rho u\right)+ \nabla_x \cdot \left( \rho u \otimes u + pI\right) = 0, \\
  &\partial_t E + \nabla_x \cdot \left( \left(E + p\right) u\right) = 0, \\
\end{aligned}
\right.
\end{equation}
where $p = \rho R T$ denotes the pressure, $I$ is the identity matrix, and $E = \rho \left(\vert u \vert ^2 + \right.$ $\left.d_v R T\right) / 2$ is the total energy.

In practice, uncertainties may arise in initial/boundary conditions and parameters, such as microscopic interaction details, due to incomplete knowledge or imprecise measurements. Recent systematic studies categorize uncertainty quantification (UQ) methods into two primary approaches: Monte Carlo (MC) \cite{dimarco2019multi,dimarco2020multiscale,caflisch1998monte,fairbanks2017low,giles2015multilevel,hu2021uncertainty,mishra2012multi,peherstorfer2018convergence,peherstorfer2016optimal} sampling and Stochastic-Galerkin (SG) \cite{hu2016stochastic,liu2020bi,shu2017stochastic,xiu2010numerical,Liu24} methods. 
While SG methods based on generalized polynomial chaos demonstrate spectral accuracy for smooth solutions, they suffer from the curse of dimensionality and require intrusive numerical implementations. 
In contrast, MC sampling methods exhibit slow convergence rates but efficiently mitigate dimensionality challenges and enable easy-to-implement non-intrusive applications. 
Prior work \cite{dimarco2019multi,dimarco2020multiscale} introduced control variate and multiple control variate techniques that leverage the multiscale nature of kinetic equations to accelerate MC convergence. Nevertheless, deterministic solutions for simplified models with large sample sizes remain computationally expensive. 
To address this limitation, using physics-informed neural networks (PINNs) to solve simplified models can be regarded as a potential alternative.

PINNs have emerged as a powerful framework for embedding physical laws into data-driven models, enabling the solution of complex differential equations through constrained optimization \cite{karniadakis2021physics,raissi2019physics}. 
By embedding observational or learning biases into loss functions, PINNs enforce adherence to governing equations such as ODEs/PDEs, offering mesh-free flexibility and robustness in high-dimensional settings. 
However, their application to multiscale systems faces critical challenges \cite{wang2022when, wang2019deep}. 
Standard PINNs often fail to preserve asymptotic consistency in singularly perturbed regimes, such as those encountered as the Knudsen number approaches zero, resulting in inaccurate macroscopic predictions. 
To address these limitations, \cite{jin2023asymptotic,jin2024ap,liu2025asymptotic,jin2022asymptotic,bertaglia2022asymptotic} have proposed modifying the loss function to incorporate asymptotic-preserving (AP) properties during training. Error estimates for such approaches have been recently obtained in \cite{wan2025error}. We refer also to \cite{Frank25, Frank23, Frank21} for related works based on the application of NN to kinetic equations.

The key contributions of this work are summarized as follows:
\begin{itemize}
    \item We extend the previously developed multiscale control variate strategy (\allowbreak MSCV) \cite{dimarco2019multi,dimarco2020multiscale}, originally designed to significantly accelerate the slow convergence of standard Monte Carlo methods for UQ, to neural network-based models. 
    Compared to traditional methods, this enhanced method achieves higher efficiency in predicting large numbers of samples.
    
    \item For the BGK equation, we propose a neural network framework based on micro-macro decomposition and conservation constraints, which preserves physical positivity and asymptotic-preserving properties, thereby enabling stable long-time simulations. 
    By correcting the Knudsen number in the BGK model and incorporating a suitable amount of Boltzmann data for supervised training, the network predictions are brought closer to the numerical solutions of the Boltzmann equation, which plays an important role in reducing the error due to uncertainty.
    
    \item Building on the theoretical framework developed in \cite{jin2024ap,liu2024discontinuity}, we extend the traditional single-sample training paradigm by introducing parameter uncertainty. 
    This yields a neural network with strong generalization capabilities, enabling the effective prediction of solutions to initial value problems under random parameters.
\end{itemize}

The remainder of this paper is structured as follows.
Section \ref{sec2} introduces fundamental concepts of MC sampling in UQ.
Section \ref{sec3} reviews the MSCV methods proposed in \cite{dimarco2019multi,dimarco2020multiscale}.
Section \ref{sec4} presents a structure- and asymptotic preserving neural network tailored for the homogeneous BGK model and subsequently extended to non-homogeneous settings.
To enhanced computational efficiency for large-scale sample predictions, Section \ref{sec4} further extends the conventional single-sample training paradigm by incorporating parameter uncertainty.
Section \ref{sec5} provides several numerical experiments for the Boltzmann equation, demonstrating the superior performance of the SAPNN-based methodology. Finally, concluding remarks are provided in Section \ref{sec6}.

\section{Standard Monte Carlo method}
\label{sec2}
Before going into the details of the presentation, let us define some notations and assumptions that will be used subsequently. If \(z \in \Omega_z\) follows the distribution \(p(z)\), the expected value is denoted by
\begin{equation*}
  \mathbb{E}[f](x,v,t) = \int_{\Omega_z} f(z, x, v, t) p(z) ~{\rm d}z,
\end{equation*}
and its variance can be written as
\begin{equation*}
  \Var(f)(x,v,t) = \int_{\Omega_z} \left(f(z, x, v, t) - \mathbb{E}[f](x,v,t)\right)^2 p(z) ~{\rm d}z.
\end{equation*}
We introduce an \(L^p\)-norm with polynomial weight as referenced in \cite{perthame1993weighted}:
\begin{equation*}
  \|f(z, \cdot, t)\|_{L_s^p(\mathbb{R}^{d_x} \times \mathbb{R}^{d_v})}^p = \int_{\mathbb{R}^{d_x} \times \mathbb{R}^{d_v}} |f(z, x, v, t)|^p (1 + |v|)^s ~{\rm d}v {\rm d}x.
\end{equation*}
For a random variable \(Z\) that takes values in \(L_s^p(\mathbb{R}^{d_x} \times \mathbb{R}^{d_v})\), we define
\begin{equation}
\label{eq1}
  \|Z\|_{L_s^p(\mathbb{R}^{d_x} \times \mathbb{R}^{d_v}; L^2(\Omega))} = \|\mathbb{E}[Z^2]^{1/2}\|_{L_s^p(\mathbb{R}^{d_x} \times \mathbb{R}^{d_v})}.
\end{equation}
The aforementioned norm, if \(p \neq 2\), differs from
\begin{equation}
\label{eq2}
  \|Z\|_{L^2(\Omega; L_s^p(\mathbb{R}^{d_x} \times \mathbb{R}^{d_v}))} = \mathbb{E} \left[ \left\|Z\right\|_{L_s^p(\mathbb{R}^{d_x} \times \mathbb{R}^{d_v})}^2 \right]^{1/2},
\end{equation}
which is typically employed in \cite{mishra2017monte,mishra2012multi}. By the Jensen inequality, it is established that
\begin{equation*}
  \|Z\|_{L_s^p(\mathbb{R}^{d_x} \times \mathbb{R}^{d_v}; L^2(\Omega))} \leq \|Z\|_{L^2(\Omega; L_s^p(\mathbb{R}^{d_x} \times \mathbb{R}^{d_v}))}.
\end{equation*}
For simplicity, we adopt norm \eqref{eq1} for \(p = 1\). This principle holds for \(p = 2\) as well, where the two norms coincide, whereas their extension to norm \eqref{eq2} for \(p = 1\) generally necessitates that \(Z\) be compactly supported \cite{dimarco2019multi,mishra2012multi}.

The standard Monte Carlo (MC) framework for UQ consists of three principal components:
\begin{enumerate}[leftmargin=*, label=\textit{Step \arabic*.}, series=mc]
    \item \textit{Stochastic Initialization}:
    Generate $L$ independent and identically distributed (i.i.d.) initial states $\{f_k^0\}_{k=1}^L$ from the random initial data $f^0$ and approximate these over the grid.
    
    \item \textit{Deterministic Evolution}:
    For each sample $f_k^0$, numerically solve the kinetic equation using appropriate discretization methods. Let $\mathcal{F}^n$ denote the numerical solution operator with time step $\Delta t^n$ satisfying CFL condition. Then the discrete-time solution satisfies:
    \begin{equation}
        f_{k}^{n} = \mathcal{F}^n(f_k^{n - 1}), \quad n = 0,1,...,N
    \end{equation}
    where the spatial and velocity discretizations are implicit in $\mathcal{F}^n$.
    
    \item \textit{Statistical Postprocessing}:
    Compute empirical estimates of solution statistics. The mean field approximation at $t^n$ is given by:
    \begin{equation}
    \label{MC}
        \E[f^n] \approx E_L[f^n] := \frac{1}{L}\sum_{k=1}^L f_{k}^n.
    \end{equation}
\end{enumerate}
The MC method demonstrates two fundamental properties that justify its prevalence in UQ. 
First, its \emph{non-intrusive} nature decouples the stochastic sampling process from the deterministic solver implementation, enabling modular integration with existing numerical schemes. 
Second, the method achieves \emph{dimension-independent convergence} with rate $O(L^{-1/2})$ \cite{dimarco2019multi,caflisch1998monte,mishra2012multi,mishra2017monte}, where $L$ denotes the sample size, irrespective of the stochastic dimension $\dim(\Omega_z)$. 
This property stems from the measure-theoretic foundation of the approach rather than geometric discretization of the parameter space. 
These characteristics collectively ensure universal applicability to broad classes of kinetic equations.

Assume that the deterministic scheme $\mathcal{F}^n$ possesses order $p$ and $q$ accuracy in the spatial and velocity directions, respectively, and that the set of random initial state $f(z, x, v, 0) = f^0(z, x, v)$ is sufficiently regular. Then, the Monte Carlo estimator defined in \eqref{MC} satisfies the following error bound \cite{dimarco2019multi}:
\begin{equation*}
  \Vert \E[f^n] - E_L[f^n] \Vert_{L_2^1(\mathbb{R}^{d_x} \times \mathbb{R}^{d_v}; L^2(\Omega))} \leq C \left(\sigma_f L^{-1/2} + \Delta x^{p}+ \Delta v^{q} \right).
\end{equation*}
Here, $\sigma_f = \Vert \Var(f)^{1/2} \Vert_{L_2^1(\mathbb{R}^{d_x} \times \mathbb{R}^{d_v})}$ and the constant $C = C(T, f^0)$ depends on the final time $T$ and initial data $f^0$.
Clearly, the sampling error in the discretization and a priori estimation can be balanced by choosing $L = \mathcal{O}(\Delta x ^{-2p})$ and $\Delta x = \mathcal{O}(\Delta v^{q/p})$.
This implies that a very large number of samples is required to achieve comparable errors, particularly when using high order deterministic solvers. 
Consequently, Monte Carlo methods may become prohibitively expensive in practical applications.

\section{Variance reduction Monte Carlo methods}
\label{sec3}
To enhance the computational efficiency of MC methods, we turn to the MSCV methods.
In this study, we employ the Boltzmann equation \eqref{kinetic}-\eqref{eq:coll} as a high-fidelity (HF) reference solution, retaining the complete collision integral term to ensure high accuracy. 
By contrast, the BGK model \eqref{kinetic} with \eqref{eq:BGK} serves as a low-fidelity (LF) solver, trading certain physical details of the collision term for enhanced computational efficiency. 
For nonhomogeneous flows, we additionally incorporate the Euler equations as a secondary LF solver to address macroscopic-scale simplifications.

\subsection{Multiscale control variate methods}
Let $\mathbb{E}[f](x, v, t)$ denote the expected solution field of the kinetic equation and set $L$ ($L \gg K$) and $K$ denote the numbers of i.i.d. samples $f_k(x,v,t)$ generated by the LF solver 
$\left(f^{\rm LF}_k\right)$ and HF solver $\left(f^{\rm HF}_k\right)$, respectively. Note that the solution of both solvers depends on the scale $\varepsilon$ defined by the Knudsen number. Moreover, both models as $\varepsilon\to 0$ formally converge towards the solution of the Euler system \eqref{Euler}.

The MSCV estimator combines both solvers as follows:
\begin{equation}
\label{estimator}
    \E[f](x,v,t) \approx E_{K, L}^{\lambda}[f](x,v,t) := \frac{1}{K}\sum_{k=1}^K f^{\rm HF}_k - \lambda\left(\frac{1}{K}\sum_{k=1}^K f^{\rm LF}_k - \frac{1}{L}\sum_{k=1}^L f^{\rm LF}_k\right),
\end{equation}
where $\lambda \in \mathbb{R}$ is a control parameter. 
As described in \cite{dimarco2019multi}, the  above control variate estimator has the following property:
\begin{lemma}
  The control variate estimator \eqref{estimator} is unbiased and consistent for any choice of $\lambda \in \mathbb{R}$. In partiular, the MSCV will reduce to MC for HF solver with $\lambda = 0$, and for $\lambda = 1$ we get
  \begin{equation*}
    E_{K, L}^{1}[f](x,v,t) = \frac{1}{L}\sum_{k=1}^L f^{\rm LF}_k + \frac{1}{K}\sum_{k=1}^K \left(f^{\rm HF}_k - f^{\rm LF}_k\right).
  \end{equation*}
\end{lemma}

Before incorporating the random variable, we assume that $L$ is sufficiently large and define
\begin{equation*}
  \mathbb{E}[f^{\rm LF}] := \lim_{L \to +\infty} \frac{1}{L}\sum_{k=1}^L f^{\rm LF}_k.
\end{equation*}
Under this assumption, the estimator in \eqref{estimator} becomes
\begin{equation}
\label{estimator2}
    \E[f](x,v,t) \approx E_{K}^{\lambda}[f](x,v,t) = \frac{1}{K}\sum_{k=1}^K f^{\rm HF}_k - \lambda\left(\frac{1}{K}\sum_{k=1}^K f^{\rm LF}_k - \mathbb{E}[f^{\rm LF}]\right).
\end{equation}
Let us consider the random variable
\begin{equation}
\label{flambda}
    f^{\lambda}(z,x,v,t) = f^{\rm HF}(z,x,v,t) - \lambda\left(f^{\rm LF}(z,x,v,t) - \mathbb{E}[f^{\rm LF}](\cdot,x,v,t)\right), 
\end{equation}
and we can quantify its variance at the point $(x, v, t)$ as
\begin{equation}
\label{vf}
  \Var(f^{\lambda}) = \Var \left( f^{\rm HF} \right) + \lambda^2 \Var \left( f^{\rm LF} \right) - 2\lambda \Cov \left( f^{\rm HF}, f^{\rm LF} \right),
\end{equation}
where $\Cov(\cdot, \cdot)$ is the covariance matrix.
We state the following theorem and provide its proof for completeness \cite{dimarco2019multi}.
\begin{lemma}
If $\Var\left( f^{\rm LF} \right) \neq 0$, the quantity
\begin{equation}
\label{lambda}
  \tilde{\lambda} = \frac{\Cov \left( f^{\rm HF}, f^{\rm LF} \right)}{\Var\left( f^{\rm LF} \right)}
\end{equation}
minimizes the variance of $f^\lambda$ at the point $(x, v, t)$ and gives
\begin{equation}
\label{vft}
  \Var \left(f^{\tilde{\lambda}}\right) = \left(1 - \tilde{\rho}^2\right) \Var \left( f^{\rm HF} \right),
\end{equation}
where $\tilde{\rho} \in [-1, 1]$ is the correlation coefficient between $f^{\rm HF}$ and $f^{\rm LF}$. 
In addition, we have
\begin{equation}
\label{limit}
  \lim_{\eps \to 0} \tilde{\lambda} (x, v, t) = 1, \quad \lim_{\eps \to 0} \Var \left( f^{\tilde{\lambda}} \right) (x, v, t) = 0, \quad \forall (x, v) \in \mathbb{R}^{d_x} \times \mathbb{R}^{d_v}.
\end{equation}
\end{lemma}
\begin{proof}
Equation \eqref{lambda} is derived by differentiating \eqref{vf} with respect to $\lambda$, where $\tilde{\lambda}$ emerges as the unique stationary point. 
The positivity of the second derivative, $2 \Var\left({f^{\rm LF}}\right) > 0$, confirms that $\tilde{\lambda}$ corresponds to a minimum. 
Substituting this optimal $\tilde{\lambda}$ back into \eqref{vf} yields \eqref{vft}, with the coefficients satisfying
\begin{equation*}
  \tilde{\rho} = \frac{\Cov \left( f^{\rm HF}, f^{\rm LF} \right)}{\sqrt{\Var\left( f^{\rm LF} \right) \Var\left( f^{\rm HF} \right)}}.
\end{equation*}
As the Knudsen number $\eps$ tends to zero, the collision term $Q(f, f)$ vanishes, and both the Boltzmann and BGK models converge to the Euler equations.
This implies $f^{\rm HF} = f^{\rm LF}$, and by applying equations \eqref{lambda} and \eqref{vft}, we obtain the limiting expression in \eqref{limit}.
\end{proof}

In practice, we can use the following formulas to compute the optimized control parameter, variance, and covariance
\begin{equation}
\label{lambda1}
\begin{aligned}
  \tilde{\lambda}_K &= \frac{\Cov_K \left( f^{\rm HF}, f^{\rm LF} \right)}{\Var_K \left( f^{\rm LF} \right)}, \\
  \Var_K(f^{\rm LF}) &= \frac{1}{K - 1}\sum_{k = 1}^K \left( f^{\rm LF}_k - \frac{1}{K} \sum_{k = 1}^K \left( f^{\rm LF}_k \right)\right)^2, \\
  \Cov_K(f^{\rm HF}, f^{\rm LF}) &= \frac{1}{K - 1}\sum_{k = 1}^K \left( f^{\rm LF}_k - \frac{1}{K} \sum_{k = 1}^K \left( f^{\rm LF}_k \right)\right) \left( f^{\rm HF}_k - \frac{1}{K} \sum_{k = 1}^K \left( f^{\rm HF}_k \right)\right).
\end{aligned}
\end{equation}
For the general case \eqref{estimator}, the total variance can be written as
\begin{equation*}
\begin{aligned}
  \Var\left( E_{K,L}^{\lambda} [f]\right) &= K^{-1}\Var\left(f^{HF} - \lambda f^{LF}\right) + L^{-1}\Var\left(\lambda f^{LF}\right) \\
  &= K^{-1}\left(\Var\left(f^{HF}\right) - 2\lambda \Cov\left(f^{LF}, f^{HF}\right)\right) \!+\! \left(K^{-1} + L^{-1}\right)\lambda^2\Var\left(f^{LF}\right)
\end{aligned}
\end{equation*}
where the central limit theorem has been used in the first line. Minimizing the above quantity with respect to $\lambda$ yields the optimal value
\begin{equation}
\label{lambda2}
  \lambda_* = \frac{L}{K + L} \tilde{\lambda}_K
\end{equation}
where $\tilde{\lambda}_K$ is given by \eqref{lambda1}.

Using the optimal value \eqref{lambda} and a deterministic solver with spatial and velocity accuracies of order $p$ and $q$, respectively, we derive the following error estimate \cite{dimarco2019multi,dimarco2020multiscale,mishra2012multi}:
\begin{equation*}
  \Vert \E[f^n] - E_{K, L}^{\tilde{\lambda}}[f^n] \Vert_{L_2^1(\mathbb{R}^{d_x} \times \mathbb{R}^{d_v}; L^2(\Omega))} \leq C \left(\sigma_{f^{\tilde{\lambda}}} K^{-1/2} + \tau_{f^{\tilde{\lambda}}} L^{-1/2} + \Delta x^{p}+ \Delta v^{q} \right),
\end{equation*}
where $C > 0$ is a constant depending on the final time and the initial data, $\sigma_{f^{\tilde{\lambda}}} = \Vert \left(1 - \tilde{\rho}^2\right)^{1/2} \Var \left( f^{\rm HF} \right)^{1/2} \Vert_{L_2^1(\mathbb{R}^{d_x} \times \mathbb{R}^{d_v})}$, and $\tau_{f^{\tilde{\lambda}}} = \Vert\tilde{\rho} \Var \left( f^{\rm HF} \right)^{1/2} \Vert_{L_2^1(\mathbb{R}^{d_x} \times \mathbb{R}^{d_v})}$.
As $\eps \to 0$, we have $\tilde{\rho} \to 1$; therefore, in the fluid limit, the statistical error reduces to that of the fine-scale control variate model.

\subsection{Multiple multiscale control variates}
Now, we extend the methodology to the use of several different fidelity solutions as control variates with the aim of further improving the variance reduction properties of MSCV methods. For consistency and clarity, we denote the algorithms that are next in fidelity to the highest fidelity solver as \( \text{LF}_i \), where \( i = 1, \ldots, I \). The multiple MSCV \cite{dimarco2020multiscale} can be written as
\begin{equation}
\label{mpestimator}
    \E[f](x,v,t) \approx E_{K}^{\Lambda}[f](x,v,t) = \frac{1}{K}\sum_{k=1}^{K} f^{\rm HF}_k - \sum_{i = 1}^I \lambda_i  \left(\frac{1}{K}\sum_{k=1}^{K} f^{{\rm LF}_i}_k - \mathbb{E}[f^{{\rm LF}_i}]\right),
\end{equation}
with 
\begin{equation*}
  \mathbb{E}[f^{{\rm LF}_i}] := \lim_{L \to +\infty} \frac{1}{L}\sum_{k=1}^{L} f^{{\rm LF}_i}_k,
\end{equation*}
and optimized control parameters $\lambda_i \in \mathbb{R}$. Similar as \eqref{flambda}, we consider the random variable
\begin{equation}
 \label{fLambda}
    f^{\Lambda}(z,x,v,t) = f^{\rm HF}(z,x,v,t) - \sum_{i = 1}^I\lambda_i \left(f^{{\rm LF}_i}(z,x,v,t) - \mathbb{E}[f^{{\rm LF}_i}]\right), 
\end{equation}
with the definition
\begin{equation}
\label{notation}
\begin{aligned}
  &\Lambda = (\lambda_1, \ldots, \lambda_I)^{\rm T}, \quad b = (b_i), \quad C = (c_{ij}), \\
  &b_i =  \Cov{(f^{\rm HF}, f^{{\rm LF}_i})}, \quad c_{ij} = \Cov{(f^{{\rm LF}_i}, f^{{\rm LF}_j})},
\end{aligned}
\end{equation}
where $C$ is the symmetric $I\times I$ covariance matrix.
Based on \eqref{fLambda} and \eqref{notation}, we can derive the variance of the random variable
\begin{equation}
\label{varfL}
  \Var{(f^{\Lambda})} = \Var{(f^{\rm HF})} + \Lambda^{\rm T} C\Lambda - 2\Lambda^{\rm T}b.
\end{equation}
For completeness, we give the following lemma \cite{dimarco2020multiscale}:
\begin{lemma}
  Assuming the covariance matrix is not singular, the vector
  \begin{equation}
  \label{sLambda}
    \tilde{\Lambda} = C^{-1}b
  \end{equation}
  minimizes the variance of $f^{\Lambda}$ at the point $(x, v, t)$ and gives
  \begin{equation}
  \label{varfLt}
    \Var{(f^{\tilde{\Lambda}})} = \left(1 - \frac{b^{\rm T}(C^{-1})^{\rm T}b}{\Var{(f^{\rm HF})}} \right) \Var{(f^{\rm HF})}.
  \end{equation}
\end{lemma}
\begin{proof}
To find the optimal values $\tilde{\lambda}_i$ for \(i = 1, \ldots, I\), we first impose the first-order optimality conditions by setting the partial derivatives of $\Var\left(f^\Lambda\right)$ with respect to $\lambda_i$ to zero:
\begin{equation*}
  \frac{\partial}{\partial \lambda_i} \Var (f^\Lambda) = 0, \quad i = 1, \ldots, I.
\end{equation*}
This results in the linear system:
\begin{equation*}
  \Cov{(f^{\rm HF}, f^{{\rm LF}_i})} = \sum_{k=1}^I \lambda_k \Cov{(f^{{\rm LF}_i}, f^{{\rm LF}_k})}, \quad i = 1, \ldots, I.
\end{equation*}
In vector-matrix form, this becomes:
\begin{equation*}
  b = C {\Lambda}.
\end{equation*}
If $C$ is nonsingular, the solution \eqref{sLambda} is uniquely determined. 
The second-order conditions confirm that $\tilde{\Lambda}$ minimizes the variance. 
Finally, substituting $\tilde{\Lambda}$ into equation \eqref{varfL} directly yields \eqref{varfLt}.
\end{proof}

To illustrate the methodology, we consider the homogeneous case with $I = 2$, where $f^{{\rm LF}_1} = f^0$, the initial data, and $f^{{\rm LF}_2} = M$, the stationary state.
A straightforward computation shows that the optimal values $\tilde{\lambda}_1$ and $\tilde{\lambda}_2$ are given by
\begin{equation}
\begin{aligned}
\label{lambda12}
  \tilde{\lambda}_1 = \frac{\Var ( M) \Cov ( f^{\rm HF}, f^0) - \Cov ( f^0, M) \Cov ( f^{\rm HF}, M)}{\Delta}, \\
  \tilde{\lambda}_2 = \frac{\Var ( f^0) \Cov ( f^{\rm HF}, M) - \Cov ( f^0, M) \Cov ( f^{\rm HF}, f^0)}{\Delta},
\end{aligned}
\end{equation}
with $\Delta = \Var ( f^0) \Var ( M) - \Cov ( f^0, M)^2 \neq 0$.

Using $K$ samples for both control variates, the optimal estimator takes the form
\begin{equation*}
\begin{aligned}
    E_{K}^{\tilde{\lambda}_1, \tilde{\lambda}_2}[f](v,t) = \frac{1}{K}\sum_{k=1}^{K} f^{\rm HF}_k(v,t) &- \tilde{\lambda}_1  \left(\frac{1}{K}\sum_{k=1}^{K} f^0_k(v) - \mathbb{E}[f^0](v)\right) \\
    &- \tilde{\lambda}_2  \left(\frac{1}{K}\sum_{k=1}^{K} M_k(v) - \mathbb{E}[M](v)\right).
\end{aligned}
\end{equation*}
At time $t = 0$, since $f^{\rm HF}(z,v,0) = f^0(z,v)$, we immediately obtain $\tilde{\lambda}_1 = 1$ and $\tilde{\lambda}_2 = 0$.
Thus, the estimator reduces to the exact expression $E_{K}^{1,0}[f](v,0) = \mathbb{E}[f^0](v)$.
As $f^{\rm HF}(z,v,t) \to M(z,v)$ for $t \to \infty$, the asymptotic behavior of the coefficients from \eqref{lambda12} yields
\begin{equation*}
  \lim_{t \to \infty} \tilde{\lambda}_1 = 0, \quad \lim_{t \to \infty} \tilde{\lambda}_2 = 1,
\end{equation*} 
and consequently, the variance of the estimator vanishes in the long-time limit:
\begin{equation*}
  \lim_{t \to \infty} E_{K}^{\tilde{\lambda}_1, \tilde{\lambda}_2}[f](v,t) = E_{K}^{0, 1}[f](v) = \mathbb{E}[M](v).
\end{equation*}

Considering $F = (F_1, \ldots, F_I)^{\rm T}$, such that $F_i = f^{{\rm LF}_i} - \mathbb{E}[f^{{\rm LF}_i}]$, $\mathbb{E} [F_i] = 0$, $i = 1, \ldots, I$, then \eqref{fLambda} can be written as
\begin{equation*}
  f^{\Lambda}(z,x,v,t) = f^{\rm HF}(z,x,v,t) - \sum_{i = 1}^I\lambda_i F_i(z,x,v,t).
\end{equation*}
Therefore, the variance \eqref{varfLt} is reduced to zero if $f$ is in the span of the set of functions  $F_1, \ldots , F_I$.
Using Gram-Schmidt orthogonalization, \eqref{varfL} can be rewritten as
\begin{equation*}
\label{varfG}
  \Var{(f^{\Gamma})} = \Var{(f^{\rm HF})} + \Gamma^{\rm T} D\Gamma - 2\Gamma^{\rm T}e,
\end{equation*}
with orthogonal basis
\begin{equation*}
  g^{{\rm LF}_i} = f^{{\rm LF}_i} - \sum_{j = 1}^{i - 1}\frac{\Cov{(g^{{\rm LF}_j}, f^{{\rm LF}_i})}}{\Var{(g^{{\rm LF}_j})}} g^{{\rm LF}_j}, \quad i = 1,\ldots I.
\end{equation*}
Here, $D$ is the diagonal matrix with elements $d_j = \Var{(g^{{\rm LF}_j})}$, $e$ is the vector with components $e_j = \Cov{(f^{{\rm HF}}, g^{{\rm LF}_j})}$, and $\Gamma = (\gamma_1, \ldots, \gamma_I)^{\rm T}$.
We give the following theorem without proof \cite{dimarco2020multiscale}
\begin{lemma}
  Given the definition $G_k = g_k - \E[g_k]$, and $L^2$ inner product
  \begin{equation*}
    \langle f, g \rangle = \int_{\Omega_z} f(z) g(z) p(z) {\rm d}z,
  \end{equation*}
  if the control variate vector $G = (G_1,\dots,G_I)^{\rm T}$ has orthogonal components, i.e. $\langle G_i, G_j \rangle = 0$ for $i \neq j$, then if $\langle G_i, G_i \rangle \neq 0$, the vector
  \begin{equation*}
    \tilde{\gamma}_k = \frac{\mathrm{Cov}(f^{\rm HF}, g_k)}{\mathrm{Var}(g_k)},\quad {k=1,\dots,I}
  \end{equation*}
  minimizes the variance of $f^{\Gamma}$ at the point $(x, v, t)$ and gives
  \begin{equation*}
    \mathrm{Var}(f^{\tilde{\Gamma}}) = \left(1 \;-\; \sum_{k=1}^I \rho_{f^{\rm HF},\,g_k}^2 \right)\,\mathrm{Var}(f^{\rm HF}),
  \end{equation*}
  where $\rho_{f^{\rm HF}, g_k} \in [-1,1]$ is the correlation coefficient between $f^{\rm HF}$ and $g_k$.
\end{lemma}

If we consider a deterministic solver with spatial and velocity accuracies of order $p$ and $q$, respectively, and assuming sufficiently regular initial data, the multiple control variate estimator
\begin{equation*}
    E_{K, L}^{\Gamma}[f](x,v,t) = \frac{1}{K}\sum_{k=1}^{K} f^{\rm HF}_k - \sum_{i = 1}^I \gamma_i  \left(\frac{1}{K}\sum_{k=1}^{K} f^{{\rm LF}_i}_k - \frac{1}{L}\sum_{k=1}^{L} f^{{\rm LF}_i}_k\right),
\end{equation*}
with the optimal values given by \eqref{sLambda}, satisfies the error bound
\begin{equation*}
  \Vert \E[f^n] - E_{K, L}^{\tilde{\Gamma}}[f^n] \Vert_{L_2^1(\mathbb{R}^{d_x} \times \mathbb{R}^{d_v}; L^2(\Omega))} \leq C \left(\sigma_{f^{\tilde{\Gamma}}} K^{-1/2} + \tau_{f^{\tilde{\Gamma}}} L^{-1/2} + \Delta x^{p}+ \Delta v^{q} \right),
\end{equation*}
with a positive constant $C$ depending on the final time and the initial data, $\sigma^2_{f^{\tilde{\Gamma}}} \! =  \! \Vert (1 - \sum_{k=1}^I \rho_{f^{\rm HF},g_k}^2 )\mathrm{Var}(f^{\rm HF}) \Vert_{L_2^1(\mathbb{R}^{d_x} \times \mathbb{R}^{d_v})}$, and $\tau^2_{f^{\tilde{\lambda}}} \!= \! \Vert \sum_{k=1}^I \rho_{f^{\rm HF},g_k}^2 \mathrm{Var}(f^{\rm HF}) \Vert_{L_2^1(\mathbb{R}^{d_x} \times \mathbb{R}^{d_v})}$.

\section{Deep neural surrogates}
\label{sec4}
Traditional deterministic numerical methods for solving kinetic equations often exhibit low computational efficiency, especially in large-scale simulations requiring multiple sample evaluations. To address this issue, we propose utilizing neural networks to learn the properties of kinetic equations, thereby facilitating the efficient generation of multiple samples. Since a trained neural network can swiftly compute solutions for given initial conditions, this approach substantially enhances computational efficiency. In this section, we provide a brief overview of the general framework of physics-informed neural networks (PINNs) \cite{karniadakis2021physics,raissi2019physics}, followed by a discussion of structure and asymptotic preserving neural networks (SAPNNs) in the context of the problems considered in this work.

\subsection{Physics-informed neural networks}
We first introduce the construction of a deep neural network (DNN), which can be characterized by the \textit{triple} $\mathcal{T} = (\mathcal{A}, \mathcal{L}, \mathcal{O})$:
\begin{enumerate}
    \item \textit{Architecture Design} ($\mathcal{A}$): The specification of the network structure, including the computational graph topology, the composition and types of layers, connectivity patterns, activation functions, and the distribution of learnable parameters.

    \item \textit{Objective Functional} ($\mathcal{L}$): The empirical risk minimization problem is formulated as follows: given a training dataset $\mathcal{D}_{\mathrm{train}} = \{(x_i, y_i)\}_{i=1}^N$,
    \[
        \min_{\theta \in \Theta} \mathcal{L}(\theta) := \frac{1}{N}\sum_{i=1}^N \ell(u_\theta(x_i), y_i),
    \]
    where $u_\theta(x)$ denotes the output of the neural network parameterized by $\theta$ for input $x$, and $\ell: \mathcal{Y}\times\mathcal{Y} \to \mathbb{R}_+$ is the loss function measuring prediction fidelity (e.g., mean squared error, cross-entropy).

    \item \textit{Parameter Optimization} ($\mathcal{O}$): The choice of optimization algorithms to solve $\min_\theta \mathcal{L}(\theta)$, typically variants of stochastic gradient descent. The parameter update rule is given by
    \[
        \theta_{k+1} = \theta_k - \eta_k \nabla_\theta \hat{\mathcal{L}}_B(\theta_k),
    \]
    where $\eta_k$ is the (possibly adaptive) learning rate, and the parameter $\hat{\mathcal{L}}_B := \frac{1}{|B|}\sum_{(x_i, y_i) \in B} \ell(u_\theta(x_i), y_i)$ denotes the mini-batch empirical loss over $B \subset \mathcal{D}_{\mathrm{train}}$.
\end{enumerate}
The generalization capability of the network refers to its performance on previously unseen data, which is estimated using a disjoint \textit{test set} $\mathcal{D}_{\mathrm{test}} = \{(x_j, y_j)\}_{j=1}^L$. The empirical test risk is computed as
\[
    \mathcal{L}_{\mathrm{test}}(\theta) \!:=  \!\frac{1}{L} \sum_{j=1}^L \ell(u_\theta(x_j), y_j),
\]
serving as an approximation to the expected (population) risk. This test error is distinct from the empirical risk $\mathcal{L}(\theta)$ evaluated on the training set.

The paradigm shift in PINNs emerges through the systematic incorporation of physical constraints. Consider a PDE system defined on spatio-temporal domain $\Omega \subset \mathbb{R}^{d} \times \mathbb{R}_+$:
\begin{equation*}
\begin{aligned}
    \mathcal{F}[u](x,t; \xi) = 0, \quad & (x,t) \in \Omega  \\
    \mathcal{B}[u](x,t; \xi) = 0, \quad & (x,t) \in \partial\Omega 
\end{aligned}
\end{equation*}
where $\mathcal{F}$ is a differential operator, $\mathcal{B}$ encodes boundary/initial conditions, $u$ represents the solution to the system, and $\xi \in \Xi$ parametrizes the physical system.
The PINN framework approximates the solution $u \approx u_\theta \in \mathcal{U}_{NN}$ through neural parametrization, where $\mathcal{U}_{NN}$ denotes the hypothesis space of DNNs. In PINN literature, the most widely used neural network architecture is the feed-forward neural network.
Setting input data $q^{(0)} \in \mathbb{R}^{d_{\text{in}}} $, the quintessential feed-forward architecture implements:
\[
    \begin{aligned}
        q^{(\ell)} &= \sigma_\ell(W^{(\ell)} q^{(\ell-1)} + b^{(\ell)}), \quad \ell=1,...,L-1 \\
        u_\theta(x) &= W^{(L)} q^{(L-1)} + b^{(L)}
    \end{aligned}
\]
with parameter set $\theta = \{W^{(\ell)}, b^{(\ell)}\}_{\ell=1}^L$, weights $W^{(\ell)} \in \mathbb{R}^{m_\ell \times m_{\ell-1}}$, bias $b^{(\ell)} \in \mathbb{R}^{m_\ell}$, and activation functions $\sigma_\ell: \mathbb{R} \to \mathbb{R}$ applied component-wise.
PINNs reformulate the learning objective by incorporating physical constraints into the loss function:
{\small
\begin{equation*}
\begin{aligned}
    \min_{\theta} \underbrace{\frac{1}{N} \sum_{i=1}^N \|u_\theta(x_i, t_i) \! - \! u_{\mathrm{obs}}(x_i, t_i)\|^2}_{\text{Data loss}}
    + \lambda \left( \! 
        \underbrace{\frac{1}{N_f} \sum_{j=1}^{N_f} \|\mathcal{F}[u_\theta](x_j, t_j)\|^2}_{\text{PDE residual loss}} 
        \! +\!  \underbrace{\frac{1}{N_b} \sum_{k=1}^{N_b} \|\mathcal{B}[u_\theta](x_k, t_k)\|^2}_{\text{Boundary residual loss}} 
    \! \right)
\end{aligned}
\end{equation*}
}
where $u_{\mathrm{obs}}(x_i, t_i)$ denotes observed data at measurement points $(x_i, t_i) \in \mathcal{D}_{\mathrm{data}}$,  and $\lambda > 0$ is a hyperparameter that balances data fidelity and physics enforcement.

\subsection{Structure and asymptotic preserving neural networks}
To rigorously capture multiscale kinetic dynamics across different propagation regimes while maintaining physical consistency, the neural network should satisfy both the asymptotic-preserving (AP) and structure-preserving (SP) properties. 
In the context of kinetic equations, the AP property ensures a unified treatment of multiscale transport phenomena, enabling seamless transitions between hyperbolic-dominated regimes (e.g., free molecular flows) and diffusion-dominated regimes (e.g., near thermodynamic equilibrium).
Simultaneously, the SP property enforces essential invariants -- such as conservation of mass, momentum, and energy, as well as the positivity of distribution functions -- to maintain numerical stability and physical fidelity in long-time simulations.
The coupled framework of SAPNN thus enables the neural network to efficiently and robustly simulate kinetic equations.

\subsubsection{Space homogeneous case: positivity and moments preservation}
For space homogeneous case, the BGK equation \eqref{kinetic} with \eqref{eq:BGK} reduces to
\begin{equation}
  \label{hkinetic}
  \partial_t f = \frac{\mu}{\eps} \left(M[f] - f\right),
\end{equation}
where $f = f(v,t)$ with the initial data $f(v,0) = f_0(v)$, we omit the random variable $z$ for brevity. 
Here, the Maxwellian distribution $M$ is strictly positive and time-independent, as it is uniquely determined by the conserved moments (mass, momentum, and energy) of the initial distribution $f_0$.
Formally, in the limit as the Knudsen number , the solution $f$ converges asymptotically to $M$.

To guarantee {strict positivity} of the kinetic density $f$, we use a structure-informed reparametrization based on the known {steady state} $M$:
\begin{equation}
\label{gf}
f = M \exp(g),\quad g=\log\left(\frac{f}{M}\right).
\end{equation}
Training the NN over $g$ presents several advantages:
\begin{itemize}
    \item Ensures $f> 0$ {by design}, avoiding the need for activation constraints.
    \item The network learns the log-correction $g$, enabling training over unconstrained outputs.
    \item Embeds {prior knowledge} from the equilibrium solution into the network structure.
\end{itemize}
Of course, the residual of the kinetic equation must be adapted.  Substituting \eqref{gf} into \eqref{hkinetic}, we derive the governing equation for $g$:
\begin{equation}
  \label{hg}
  \partial_t g = \frac{\mu}{\eps} \left(\exp{(-g)} - 1\right).
\end{equation}
In the vanishing Knudsen number limit, a dominant balance argument implies $g \to 0$.
Consequently, the relation \eqref{gf} yields
\begin{equation*}
  \lim_{\eps \to 0}f = \lim_{\eps \to 0} M \exp{(g)} = M,
\end{equation*}
confirming the convergence to the equilibrium Maxwellian distribution.

We consider $g^{NN}_{\theta}(v,t)$ to be a deep neural network with inputs $v$, $t$ and trainable parameters $\theta$, to approximate the solution of our system: $g(v, t) \approx g^{NN}_{\theta}(v,t)$.
Moreover, we define the kinetic density approximated by a deep neural network as 
\begin{equation*}
  f^{NN}_{\theta}(v,t) := M(v, t) \exp{(g^{NN}_{\theta}(v,t))}.
\end{equation*}
Based on the transformation \eqref{gf}, we restrict the neural network approximation $g^{NN}_{\theta}(v,t)$ to satisfy the physics imposed by the residual
\begin{equation*}
  \mathcal{L}^\varepsilon_r(\theta) = \sum_{n = 1}^{N_r}  \left\vert {\eps}\partial_t g^{NN}_{\theta}(v^n_r,t^n_r) - \mu\left(\exp{\left(-g^{NN}_{\theta}(v^n_r,t^n_r)\right)} - 1\right) \right\vert^2,
\end{equation*}
on a finite set of $N_r$ user-specified scattered points inside the domain, $\{(v^n_r,t^n_r)\}_{n = 1}^{N_r} \subset \Omega$ (referred as residual points) and we also enforce the initial and space-boundary conditions of the system on $N_b$ scattered points of the space-time boundary $\mathcal{B}(f (v, t))$, i.e. $\{(v^k_b,t^k_b)\}_{k = 1}^{N_b} \subset \partial \Omega$ \cite{karniadakis2021physics,bertaglia2022asymptotic}. 
We also consider to have access to measured data, with a dataset $\{(f^i_d,v^i_d,t^i_d)\}_{i = 1}^{N_d}$, with $f^i_d = f(v^i_d,t^i_d)$, available in a finite set of fixed training points.
The losses of initial and space-boundary conditions $\mathcal{L}_b(\theta)$, and measured data $\mathcal{L}_d(\theta)$ have the detailed expression
\begin{equation*}
\begin{aligned}
  \mathcal{L}_b(\theta) &= \sum_{k = 1}^{N_b}  \left\vert  f^{NN}_{\theta}(v^k_b,t^k_b) - f(v^k_b,t^k_b)\right\vert^2, \\
  \mathcal{L}_d(\theta) &= \sum_{i = 1}^{N_d}  \left\vert  f^{NN}_{\theta}(v^i_d,t^i_d) - f(v^i_d,t^i_d)\right\vert^2. \\
\end{aligned}
\end{equation*}
Moreover, considering the points $\{(v_m^{j,l},t_m^j)\}_{j = 1, l= 1}^{N_m, N_j} \subset \Omega$, we define the loss of moment
{\small
\begin{equation*}
\begin{aligned}
  \mathcal{L}_m(\theta) = \sum_{j = 1}^{N_m} \Bigg( \Bigg\vert \rho_0 - \sum_{l = 1}^{N_j} \omega^*_l f^{NN}_{\theta}(v_m^{j,l},t_m^j)  \Bigg\vert^2 &+ \Bigg\vert \rho_0 u_0 - \sum_{l = 1}^{N_j} \omega^*_l v_m^{j,l} f^{NN}_{\theta}(v_m^{j,l},t_m^j)  \Bigg\vert^2  \\
  & + \Bigg\vert E_0 - \sum_{l = 1}^{N_j} \omega^*_l \frac{\vert v_m^{j,l} \vert^2}{2} f^{NN}_{\theta}(v_m^{j,l},t_m^j)  \Bigg\vert^2 \Bigg), 
\end{aligned}
\end{equation*}
}
to enforce the conservation of mass, momentum, and energy.
More specifically, when considering the 1D velocity domain, we truncate it as 
\(\Omega_v = [v_{\min}, v_{\max}]\) and apply a uniform discretization with $N_v = N_l$ points. In this case, the weights $\omega_l^*$ in the discrete formulation can be obtained via the trapezoidal rule.
Thus, in the training process of the PINN, we minimize the following SP-AP loss function, composed of four mean squared error terms
\begin{equation}
  \mathcal{L}^\varepsilon(\theta) = \omega_m \mathcal{L}_m(\theta) + \omega_r \mathcal{L}_r(\theta) + \omega_b \mathcal{L}_b(\theta) + \omega_d \mathcal{L}_d(\theta)
\end{equation}
where $\omega_m, \omega_r, \omega_b$, and $\omega_d$ characterize the weights associated to each contribution.
Neural networks relying solely on physical coordinate inputs can only simulate systems under a single set of initial-boundary values. To extend this approach to multiple samples and achieve generalization for the ordinary differential equation (ODE) \eqref{hg}, we simply replace the grid coordinate input $v$ with the initial value $g_0$.
The framework for the homogeneous BGK equation is shown in Fig.~\ref{SANNH}.

\begin{figure}[htb]
    \centering
  \begin{tikzpicture}[x=1.0cm, y=1.0cm, >=Stealth, 
      every node/.style={font=\sffamily},
      every path/.style={line width=1pt},
      inputnode/.style={draw=blue!70, thick, circle, minimum size=0.3cm, font=\sffamily\footnotesize},
      hiddencircle/.style={draw=cyan!60!black, thick, circle, minimum size=0.3cm, font=\sffamily\footnotesize},
      outputnode/.style={draw=blue!70, thick, circle, minimum size=0.3cm, font=\sffamily\footnotesize},
      arrow1/.style={-Stealth, line width = 0.6pt},
      arrow2/.style={-Stealth, line width = 0.5pt},
      dashbox/.style={draw=cyan!60!black, rounded corners, dashed, thick, inner sep=6pt, fill=none},
      rect/.style={draw=orange!100, rounded corners, thick, font=\sffamily\footnotesize}
      ]

  \node[inputnode, inner sep=2pt] (g0) {$g_0$};
  \node[inputnode, below= 0.34cm of g0] (t) {$t$};

  \foreach \i in {1,2,3,4}
      \node[hiddencircle, right=0.9cm of g0, yshift={(\i-3)*0.9cm}] (h1\i) {$\sigma$};

  \foreach \i in {1,2,3,4}
      \node[hiddencircle, right=0.5cm of h13, yshift={(\i-3)*0.9cm}] (h2\i) {$\sigma$};

  \foreach \i in {1,2,3,4}
      \node[hiddencircle, right=0.5cm of h23, yshift={(\i-3)*0.9cm}] (h3\i) {$\sigma$};

    \coordinate (target) at ($(h33)!0.5!(h32)$);
    \node[outputnode, right=0.7cm of target,inner sep=3pt] (g) {$g$};

  \foreach \i in {1,2,3,4} {
      \draw[arrow1, black!100] (g0.east) -- (h1\i.west);
      \draw[arrow1, black!100] (t.east) -- (h1\i.west);
  }

  \foreach \i in {1,2,3,4} {
      \foreach \j in {1,2,3,4} {
          \draw[arrow2, gray!60] (h1\i.east) -- (h2\j.west);
          \draw[arrow2, gray!60] (h2\i.east) -- (h3\j.west);
      }
      \draw[arrow1, black!100] (h3\i.east) -- (g.west);
  }


  \node[fit=(h14) (h31), dashbox ] (boxnn) {};

  \node[draw=gray!70, thick, circle, minimum size=0.5cm, inner sep=0.8pt,  right=0.8cm of g, yshift=0.45cm, font=\sffamily\footnotesize] (f) {$f\!\!=\!\!Me^{g}$};
  \node[draw=green!60!black, thick, circle, minimum size=0.5cm, above=0.8cm of g, inner sep=2pt, font=\sffamily\footnotesize] (M) {$M$};

  \draw[arrow1] (M.east) -- (f.west);
  \draw[arrow1] (g.east) -- (f.west);

  \node[draw=gray!70, thick, circle, minimum size=0.5cm, inner sep=0.9pt, right=0.8cm of g, yshift=-1.35cm] (dt) {$\frac{\partial}{\partial t}$};


  \draw[arrow1] (g.east) -- (dt.west);

  \node[rect, right=4.6cm of h34, minimum width=1cm, minimum height=0.6cm, anchor=center] (momentloss) {$\mathcal{L}_{m}$};
  \node[rect, right=4.6cm of h33, minimum width=1cm, minimum height=0.6cm, anchor=center] (dataloss) {$\mathcal{L}_{d}$};
  \node[rect, right=4.6cm of h32, minimum width=1cm, minimum height=0.6cm, anchor=center] (icloss) {$\mathcal{L}_{b}$};
  \node[rect, right=4.6cm of h31, minimum width=1cm, minimum height=0.6cm, anchor=center] (pdeloss) {$\mathcal{L}_{r}$};

  \node[fit=(momentloss) (pdeloss), dashbox, inner sep=4pt, draw=orange!100] (boxnn) {};

  \draw[arrow1, draw=orange!70!black] (f.east) -- (momentloss.west);
  \draw[arrow1, draw=orange!70!black] (f.east) -- (dataloss.west);
  \draw[arrow1, draw=orange!70!black] (f.east) -- (icloss.west);
  \draw[arrow1, draw=orange!70!black] (dt.east) -- (pdeloss.west);

  \node[rect, right=2cm of $(dataloss)!0.5!(momentloss)$, minimum width=1cm, minimum height=0.6cm, anchor=center] (totalloss) {$\mathcal{L}^\varepsilon$};

  \draw[arrow1, draw=orange!70!black] ($(dataloss)!0.5!(momentloss) + (0.63,0)$) -- (totalloss.west);

  \node[rect, draw=yellow!80!black, below=1.8cm of totalloss, minimum width=1cm, minimum height=0.6cm, anchor=center] (theta) {$\theta$};

  \draw[arrow1, draw=yellow!40!black] (totalloss.south) -- (theta.north);

  \node[left=-0.2cm of totalloss, yshift=-0.35cm, font=\sffamily\footnotesize, rotate=90] (Min) {Minimize};

  \end{tikzpicture}

  \caption{SAPNN framework for homogeneous BGK equation preserving positivity and enforcing physical conservations.}
    \label{SANNH}
\end{figure}
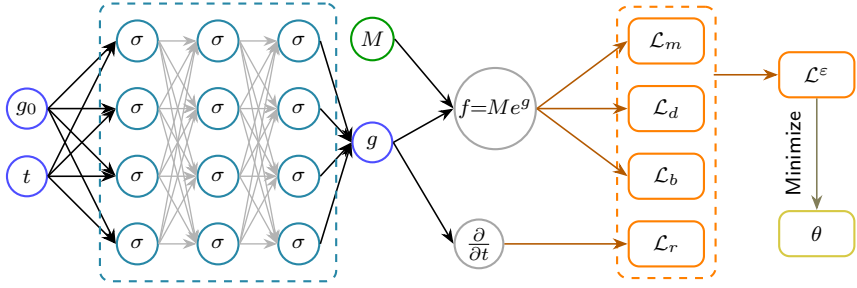

\begin{remark}
  In the case of homogeneous BGK equations, \eqref{hg} essentially reduces to a system of ODEs, whose temporal evolution is governed solely by the initial conditions and the parameter $\eps$. Following the existence and uniqueness theorem for ODEs, it suffices to ensure that the training samples adequately cover the principal support region of the initial condition's probability density function, without explicitly incorporating the stochastic coordinate $z$ as an input parameter in the neural network. However, for nonhomogeneous scenarios, the governing equation transforms into a partial differential equation containing spatial derivative terms. This structural change introduces characteristic-dependent coupling between solutions at adjacent spatial locations. To properly capture such spatial dependencies and enhance the model's generalization capability under varying initial-boundary conditions, it becomes necessary to explicitly include the stochastic variable $z$ as an additional input parameter in the neural network architecture.
\end{remark}

\subsubsection{Space non-homogeneous case: the asymptotic preserving property}
To ensure theoretical completeness, we extend the methodology proposed in \cite{jin2024ap} to multi-scale scenarios where the AP property becomes essential. 
Following the theoretical framework in \cite{jin2011class}, the BGK equation \eqref{kinetic} and its moment equations \eqref{moment} can be jointly formulated as a coupled system:
\begin{equation*}
\left\{
\begin{aligned}
  &\eps (\partial_t f + v \cdot \partial_x f) = \mu (M[U] - f), \\
  &\partial_t U + \nabla_x \cdot \int_{\mathbb{R}^{d_v}} v f \phi(v) ~{\rm d}v = 0, \\
  & U = \int_{\mathbb{R}^{d_v}} f \phi(v) ~{\rm d}v.
\end{aligned}
\right.
\end{equation*}
where the distribution function $f$ asymptotically converges to the local Maxwellian $M[U]$ causing the second equation to reduce to the compressible Euler system \eqref{Euler}.  An asymptotic-preserving neural network is build in order to satisfy the same property as sketched in Figure \ref{fg:AP}.


\begin{figure}[ht]
  \centering
  \begin{tikzpicture}[>=Latex,semithick,node distance=6.5cm]
    \node (tl) {$\mathcal{R}_{\mathrm{NN}}(\mathcal{F}^{\varepsilon})$};
    \node (tr) [right=6.5cm of tl] {$\mathcal{R}_{\mathrm{NN}}(\mathcal{F}^{0})$};
    \node (bl) [below=3.2cm of tl] {$\mathcal{F}^{\varepsilon}$};
    \node (br) [below=3.2cm of tr] {$\mathcal{F}^{0}$};

    \draw[->] (tl) -- node[above] {$\varepsilon \to 0$} (tr);
    \draw[->] (bl) -- node[above] {$\varepsilon \to 0$} (br);
    \draw[->] (tl) -- node[left ] {$\mathcal{R}_{\mathrm{NN}} \to 0$} (bl);
    \draw[->] (tr) -- node[right] {$\mathcal{R}_{\mathrm{NN}} \to 0$} (br);
  \end{tikzpicture}

  \caption{AP framework illustration. The symbol $\mathcal{F}^{\varepsilon}$ stands for the $\varepsilon$-dependent multiscale system (BGK); $\mathcal{F}^{0}$ denotes its limit representation (Euler), independent of $\varepsilon$. Training imposes the residual $\mathcal{R}_{\mathrm{NN}}(\mathcal{F}^{\varepsilon})=\mathcal{R}^{\varepsilon}_{\mathrm{NN}}$; in the limit $\varepsilon\to 0$ we obtain $\mathcal{R}_{\mathrm{NN}}(\mathcal{F}^{0})=\mathcal{R}^{0}_{\mathrm{NN}}$. The network is said to be AP when $\mathcal{R}_{\mathrm{NN}}(\mathcal{F}^{\varepsilon})$ adequately reproduces $\mathcal{F}^{0}$.}
  \label{fg:AP}
\end{figure}

In \cite{jin2024ap}, four separate neural networks were constructed: one sub-network $f_{\theta}^{NN}$ for the BGK model, and three additional sub-networks ${\rho}_{\theta}^{NN}$, ${u}_{\theta}^{NN}$, and ${T}_{\theta}^{NN}$ to model density, velocity, and temperature components, with inputs $(x, v, t)$.
In this paper, we only design two deep neural networks, one $g_{\theta}^{NN}$ for the BGK model, and another one $\tilde{W}_{\theta}^{NN} = (\tilde{\rho} , \tilde{u}, \tilde{T})_{\theta}^{NN}$, with inputs $(x, t, z)$.
Specifically, in contrast to prior architectures that explicitly depend on velocity variable $v$, our proposed BGK model network $g_{\theta}^{NN}$ employs a structured dimension-reduction strategy, taking only physical coordinates $x$, time $t$, and stochastic perturbation $z$ as inputs (without explicit velocity dependence). 
The network outputs a discretized distribution function vector defined on a uniform grid in the 2D velocity space:
\begin{equation*}
  g_{\theta}^{\rm NN}(x,t,z) = \left( g_1, g_2, \ldots, g_{N_l} \right)^{\top} \in \mathbb{R}^{N_l},
\end{equation*}
where $N_l = N_{v1} \times N_{v2}$ denotes the total number of grid points after velocity-space discretization. 
By decoupling the explicit dependence on $v$, this design reduces the input dimension from $(x, v, t, z) \in \mathbb{R}^{d_x + d_v+d_z + 1}$ to $(x, t, z) \in \mathbb{R}^{d_x +d_z + 1}$, thereby achieving greater computational efficiency in large-scale multi-sample training while preserving physical interpretability.

In the space non-homogeneous case, the Maxwellian distribution $M$, which is no longer at steady state, depends on both time $t$ and physical space $x$.
Substituting equation \eqref{gf} into \eqref{kinetic}, we obtain the following form of the non-homogeneous BGK equation:
\begin{equation}
\label{complexg}
  M (\partial_t g + v \cdot \nabla_x g) + \partial_t M + v \cdot \nabla_x M = \frac{\mu}{\eps} M (\exp{(-g)} - 1),
\end{equation}
or equivalentely
\begin{equation*}
  \partial_t g + v \cdot \nabla_x g + \partial_t \log(M) + v \cdot \nabla_x \log(M) = \frac{\mu}{\eps}(\exp{(-g)} - 1).
\end{equation*}
There are two approaches to training this equation. 
One approach is to construct a separate model to capture the evolution of $M$, which requires significant computational resources.
Alternatively, $M$ can be treated as a function of the macroscopic variables $W = (\rho, u, T)$, allowing $\partial_t M$ and $\partial_x M$ to be expressed in terms of $\partial_t W$ and $\partial_x W$.
While this approach is more computationally efficient, it introduces additional complexity to the equation and may reduce the stability of the simulation.

To simplify the computation, in this work we consider the transformation $f = \exp{(g)}$ for the non-homogeneous BGK.
Under this transformation, equation \eqref{kinetic} becomes
\begin{equation*}
  \partial_t g + v \cdot \nabla_x g = \frac{\mu}{\eps} (\exp{(-g)} - 1),
\end{equation*}
which is significantly simpler than the original form \eqref{complexg}.
To ensure the positivity of the predicted density $\rho^{NN}_\theta$ and temperature $T^{NN}_\theta$, we adopt the technique proposed in \cite{jin2024ap}, where
\begin{equation*}
  \rho^{NN}_\theta = \log{(1 + \exp{(\tilde{\rho}^{NN}_\theta)})}, \quad T^{NN}_\theta = \log{(1 + \exp{(\tilde{T}^{NN}_\theta)})}.
\end{equation*}

The SAPNN empirical risk for the system of BGK equation is
\begin{equation}
  \mathcal{L}^\varepsilon(\theta)=\omega_m \mathcal{L}_m(\theta) + \omega_{r1} \mathcal{L}^\varepsilon_{r1}(\theta)+ \omega_{r2} \mathcal{L}_{r2}(\theta) + \omega_b \mathcal{L}_b(\theta) + \omega_d \mathcal{L}_d(\theta)
\end{equation}
where 
\begin{equation*}
\begin{aligned}
  \mathcal{L}_m(\theta) &= \sum_{h = 1}^{3} \sum_{j = 1}^{N_m}  \left\vert U_{\theta,h}^{NN}(x_m^{j},t_m^j, z_m^j) - \sum_{l = 1}^{N_l} \omega_l^* f^{NN}_{\theta, l}(x_m^{j},t_m^j, z_m^j) \phi_h(v_l)  \right\vert^2, \\
  \mathcal{L}^\varepsilon_{r1}(\theta) &= \sum^{N_{r1}}_{n1 = 1}\sum^{N_{l}}_{l = 1} \left\vert \left( \eps \left(\partial_t g^{NN}_{\theta, l}  + v_l \cdot \partial_x g^{NN}_{\theta, l}  \right) - \mu\left(\exp{\left( -  g^{NN}_{\theta, l}\right)} - 1\right)\right) (x_{r1}^{n1},t_{r1}^{n1}, z_{r1}^{n1})  \right\vert^2\\
  \mathcal{L}_{r2}(\theta) &= \sum_{h = 1}^{3} \sum_{n2 = 1}^{N_{r2}}  \left\vert \partial_t U_{\theta,h}^{NN}(x_{r2}^{n2},t_{r2}^{n2}, z_{r2}^{n2}) + \partial_x \left(\sum_{l = 1}^{N_l} \omega_l^*  v_l f^{NN}_{\theta, l}(x_{r2}^{n2},t_{r2}^{n2}, z_{r2}^{n2}) \phi_h(v_l)  \right)  \right\vert^2, \\ 
  \mathcal{L}_b(\theta)  &= \sum_{k = 1}^{N_b} \left(\sum^{N_{l}}_{l = 1} \left\vert  \left(f^{NN}_{\theta,l} - f_{l}\right) (x^k_b,t^k_b,z^k_b) \right\vert^2 + \sum^{3}_{h = 1}\left\vert  \left(U^{NN}_{\theta,h} - U_{h}\right) (x^k_b,t^k_b,z^k_b) \right\vert^2 \right), \\
  \mathcal{L}_d(\theta) &= \sum_{i = 1}^{N_d} \left(\sum^{N_{l}}_{l = 1} \left\vert  \left(f^{NN}_{\theta,l} - f_{l}\right) (x^i_d,t^i_d,z^i_d) \right\vert^2 + \sum^{3}_{h = 1}\left\vert  \left(U^{NN}_{\theta,h} - U_{h}\right) (x^i_d,t^i_d,z^i_d) \right\vert^2 \right),
\end{aligned}
\end{equation*}
with different weights $\omega_m, \omega_{r1}, \omega_{r2}, \omega_b$, $\omega_d$ and $\phi_h(v)$ stands for the $h$-th collision invariant used to form the moment equations.
Here, the superscripts and subscripts of the coordinates $x, t$, and $z$ remain the same as previously described and will not be repeated for the sake of brevity.
We sketch the framework of the SAPNN for the multiscale space non-homogeneous BGK model in Fig.~\ref{SANNNH}.
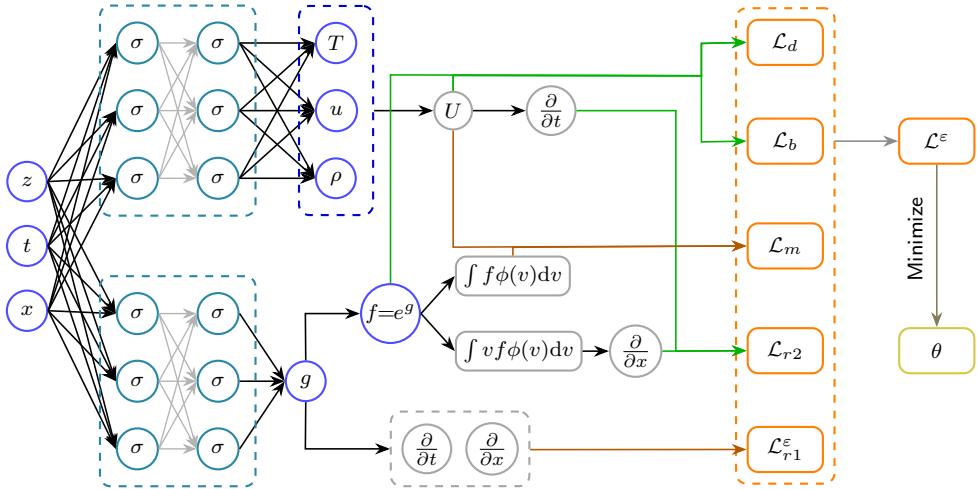
\begin{figure}[htb]
    \centering
  \begin{tikzpicture}[x=1.0cm, y=1.0cm, >=Stealth, 
      every node/.style={font=\sffamily},
      every path/.style={line width=1pt},
      inputnode/.style={draw=blue!70, thick, circle, minimum size=0.3cm, font=\sffamily\footnotesize},
      hiddencircle/.style={draw=cyan!60!black, thick, circle, minimum size=0.3cm, font=\sffamily\footnotesize},
      outputnode/.style={draw=blue!70, thick, circle, minimum size=0.3cm, font=\sffamily\footnotesize},
      arrow1/.style={-Stealth, line width = 0.6pt},
      arrow2/.style={-Stealth, line width = 0.5pt},
      dashbox/.style={draw=cyan!60!black, rounded corners, dashed, thick, inner sep=6pt, fill=none},
      rect/.style={draw=orange!100, rounded corners, thick, font=\sffamily\footnotesize}
      ]

  \node[inputnode] (t) {$t$};
  \node[inputnode, below= 0.3cm of t] (x) {$x$};
  \node[inputnode, above= 0.3cm of t] (z) {$z$};

  \foreach \i in {1,2,3}
      \node[hiddencircle, right=0.9cm of t, yshift={(\i)*0.9cm}] (a1\i) {$\sigma$};

  \foreach \i in {1,2,3}
      \node[hiddencircle, right=0.9cm of t, yshift={-(\i)*0.9cm}] (b1\i) {$\sigma$};

  \foreach \i in {1,2,3}
      \node[hiddencircle, right=0.5cm of a1\i] (a2\i) {$\sigma$};

  \foreach \i in {1,2,3}
      \node[hiddencircle, right=0.5cm of b1\i] (b2\i) {$\sigma$};

  \node[outputnode, right=1cm of a21, inner sep=3pt] (rho) {$\rho$};
  \node[outputnode, right=1cm of a22, inner sep=3.4pt ] (u) {$u$};
  \node[outputnode, right=1cm of a23, inner sep=2.5pt] (T) {$T$};
  \node[outputnode, right=0.6cm of b22, inner sep=3pt ] (g) {$g$};

  \foreach \i in {1,2,3} {
      \draw[arrow1, black!100] (z.east) -- (a1\i.west);
      \draw[arrow1, black!100] (t.east) -- (a1\i.west);
      \draw[arrow1, black!100] (x.east) -- (a1\i.west);

      \draw[arrow1, black!100] (z.east) -- (b1\i.west);
      \draw[arrow1, black!100] (t.east) -- (b1\i.west);
      \draw[arrow1, black!100] (x.east) -- (b1\i.west);
  }

  \foreach \i in {1,2,3} {
      \foreach \j in {1,2,3} {
          \draw[arrow2, gray!60] (a1\i.east) -- (a2\j.west);
          \draw[arrow2, gray!60] (b1\i.east) -- (b2\j.west);
      }
      \draw[arrow1, black!100] (b2\i.east) -- (g.west);
      \draw[arrow1, black!100] (a2\i.east) -- (rho.west);
      \draw[arrow1, black!100] (a2\i.east) -- (T.west);
      \draw[arrow1, black!100] (a2\i.east) -- (u.west);
  }


  \node[fit=(a11) (a23), dashbox ] (boxnn) {};
  \node[fit=(b11) (b23), dashbox ] (boxnn) {};
  \node[fit=(T) (rho), dashbox, draw=blue!80!black ] (boxnn) {};

  \node[outputnode, right=1.6cm of b21, inner sep=0.8pt] (f) {$f\!\!=\!\!e^{g}$};
  \node[draw=gray!70, thick, circle, minimum size=0.5cm, right=1cm of u, inner sep=2pt, font=\sffamily\footnotesize] (U) {$U$};
  \node[draw=gray!70, thick, circle, minimum size=0.5cm, inner sep=0.9pt, right=2.15cm of b23] (dt) {$\frac{\partial}{\partial t}$};
  \node[draw=gray!70, thick, circle, minimum size=0.5cm, inner sep=0.9pt, right=3cm of b23] (dx) {$\frac{\partial}{\partial x}$};

  \draw[arrow1] (g.north) --($(b21)+(1.16, 0)$) -- (f.west);
  \draw[arrow1] (g.south) --($(b23)+(1.16, 0)$) -- ($(dt.west)-(0.15, 0)$);
  \draw[arrow1] ($(u)+(0.5, 0)$) -- (U.west);

  \node[draw=gray!70, thick, circle, minimum size=0.5cm, inner sep=0.9pt, right=0.7cm of U] (dt2) {$\frac{\partial}{\partial t}$};

  \node[fit=(dx) (dt), dashbox, inner sep=4pt, draw=gray!70!] (boxnn) {};

  \draw[arrow1] (U.east) -- (dt2.west);

  \node[rect, draw=gray!70, inner sep=2.9pt, right=0.45cm of f, yshift = 0.5cm] (Uf) {$\int f\phi(v) {\rm d}v$};
  \node[rect, draw=gray!70, inner sep=2.9pt, right=0.45cm of f, yshift = -0.5cm] (Ff) {$\int vf\phi(v) {\rm d}v$};
  \node[draw=gray!70, thick, circle, minimum size=0.5cm, inner sep=0.9pt, right=0.35cm of Ff] (dx2) {$\frac{\partial}{\partial x}$};
  \draw[arrow1] (f.east) -- (Uf.west);
  \draw[arrow1] (f.east) -- (Ff.west);
  \draw[arrow1] (Ff.east) -- (dx2.west);

  \node[rect, right=9.8 of t, minimum width=1cm, minimum height=0.6cm, anchor=center] (closs) {$\mathcal{L}_{m}$};
  \node[rect, below=1.08cm of closs, minimum width=1cm, minimum height=0.6cm, anchor=center] (pde2loss) {$\mathcal{L}_{r2}$};
  \node[rect, below=1cm of pde2loss, minimum width=1cm, minimum height=0.6cm, anchor=center] (pde1loss) {$\mathcal{L}^{\varepsilon}_{r1}$};
  \node[rect, above=1.08cm of closs, minimum width=1cm, minimum height=0.6cm, anchor=center] (bloss) {$\mathcal{L}_{b}$};
  \node[rect, above=1cm of bloss, minimum width=1cm, minimum height=0.6cm, anchor=center] (dloss) {$\mathcal{L}_{d}$};

  \node[fit=(pde1loss) (dloss), dashbox, inner sep=4pt, draw=orange!100] (boxnn) {};

  \draw[arrow1, draw=orange!70!black] ($(dx.east)+(0.15,-0.01)$) -- (pde1loss.west);

  \draw[arrow1, draw=orange!70!black] (U.south) -- ($(U.south)+(0, -1.552)$) -- (closs.west);
  \draw[arrow1, draw=orange!70!black] (Uf.north) -- ($(Uf.north)+(0, 0.11)$) -- (closs.west);

  \draw[arrow1, draw=green!70!black] (dx2.east) -- (pde2loss.west);
  \draw[arrow1, draw=green!70!black] (dt2.east) -- ($(dt2.east) + (1.32,0)$) -- ($(dt2.east) + (1.32,-3.2)$) -- (pde2loss.west);

  \draw[arrow1, draw=green!70!black] (U.north) -- ($(U.north) + (0, 0.205)$) -- ($(U.north) + (3.3, 0.2)$) -- ($(U.north) + (3.3, 0.63)$) -- (dloss.west);
  \draw[arrow1, draw=green!70!black] (U.north) -- ($(U.north) + (0, 0.205)$) -- ($(U.north) + (3.3, 0.2)$) -- ($(U.north) + (3.3, -0.68)$) -- (bloss.west);
  \draw[arrow1, draw=green!70!black] (f.north) -- ($(f.north) + (0, 2.77)$) -- ($(U.north) + (3.3, 0.2)$) -- ($(U.north) + (3.3, 0.63)$) -- (dloss.west);

  \node[rect, right=1.5cm of bloss, minimum width=1cm, minimum height=0.6cm, anchor=center] (totalloss) {$\mathcal{L}^\varepsilon$};

  \draw[arrow1, draw=gray!90] ($(bloss.east)+(0.14,0)$) -- (totalloss.west);

  \node[rect, draw=yellow!80!black, right=1.5cm of pde2loss, minimum width=1cm, minimum height=0.6cm, anchor=center] (theta) {$\theta$};

  \draw[arrow1, draw=yellow!40!black] (totalloss.south) -- (theta.north);

  \node[right=-0.8cm of totalloss, yshift=-2.cm, font=\sffamily\footnotesize, rotate=90] (Min) {Minimize};

  \end{tikzpicture}

  \caption{SAPNN framework for nonhomogeneous BGK equation. In addition to positivity and conservation the network enjoys the AP property.}
    \label{SANNNH}
\end{figure}

\section{Numerical examples}
\label{sec5}
We divide the field $\Omega_z$ of random data $z$ into $N_1$ equal cells, and consider $N_2$ Gauss-Lobatto points for each cell, with $z_{ij}$ stands for the $j$-th Gauss-Lobatto point in cell $i$. 
Therefore, the following results solved deterministically using classical numerical methods for Boltzmann equation can be regarded as a reference solution for expected value
\begin{equation*}
\label{fGL}
  \mathbb{E}_{GL}[f] = \frac{1}{N_1} \sum_{i = 1}^{N_1} \sum_{j = 1}^{N_2} \omega_j f(z_{ij})
\end{equation*}
where $\omega_j$, $j = 1, 2, \ldots, N_2$, are the Gauss-Lobatto weights.

We emphasize that the BGK samples are trained and generated using the SAPNN framework, whereas the Euler samples are obtained by training with the PINN framework \cite{liu2024discontinuity}, incorporating an additional input $z$ and monitoring data. In contrast, the Boltzmann samples are produced using the fast deterministic spectral method \cite{mouhot2006fast}.

In the following numerical example, $f_B$, $f_E$, and $f_{Bolt}$ denote the distribution functions corresponding to the BGK model, the Euler model, and the Boltzmann model, respectively.
Let SAPNN($f_B$) and PINN($f_E$) denote the numerical solution computed by applying a variance-reduced Monte Carlo method to the approximation BGK model and Euler model, respectively.
SAPNN($f_I$, $f_J$) represents the solution enhanced by a multiple variance reduction technique applied jointly to two functions $f_I$ and $f_J$.
The SAPNN is trained in two stages. In the first stage, we use the Adam optimizer (learning rate $10^{-3}$) with a ReduceLROnPlateau scheduler (patience = 250, factor = 0.8) for 12000 epochs. In the second stage, we switch to the L-BFGS optimizer (learning rate = 0.1, history size = 100, maximum iterations = 20, strong--Wolfe line search) for 1500 epochs. Training is performed on two GPUs in parallel. For homogeneous cases, GPU memory usage is about 3 GB, with a runtime of 8-15 minutes per case. For non-homogeneous cases, GPU memory usage increases to 26 GB, with a runtime of 3-5 hours per case.

\subsection{Entropic calibration of BGK equations}
One essential aspect in preventing inconsistencies between the high-fidelity Boltzmann data and the low-fidelity BGK surrogate model during training is to ensure that the two models are as closely aligned as possible over the time interval of interest. Since the BGK model is an approximation of the Boltzmann equation with a simplified collision operator, its accuracy heavily depends on the choice of the collision frequency $\mu$, which governs the rate at which the system relaxes toward equilibrium.

To systematically calibrate this parameter and improve the consistency between the two models, we adopt a data-driven strategy based on minimizing the discrepancy in entropy evolution, where these data come from a traditional numerical solver. Specifically, for a given set of Boltzmann data, we estimate the optimal relaxation frequency $\mu^*$ 
  by minimizing the error in the entropy functional, defined in \eqref{eq:H}
which is a classical Lyapunov functional for both the Boltzmann and BGK equations in the homogeneous setting and is known to be monotonically decreasing in time due to entropy dissipation.

In this context, the entropy serves as a physically meaningful scalar quantity that reflects the underlying kinetic relaxation dynamics. By comparing the entropy evolution computed from the BGK model, $H(f_B)$, with that of the reference Boltzmann solution, $H(f_{\text{Bolt}})$, we define a cost functional 
\[
\mathcal{J}(\mu)=\| H(f_B) - H(f_{\text{Bolt}}) \|,
\]
which acts as an indicator of the dynamic mismatch between the two models. 

Minimizing this entropy-based discrepancy with respect to $\mu$ yields an {optimal relaxation parameter} $\mu^*$ that makes the BGK model best reproduce the macroscopic behavior of the Boltzmann system in terms of entropy dissipation. The corresponding calibrated BGK solution is denoted by $f_B^*$, and will be used throughout the training process as the low-fidelity surrogate aligned with the high-fidelity data.
\begin{remark}
In principle, one could consider the relaxation parameter $\mu$ as an additional trainable variable of the neural network, and include the entropy-based discrepancy directly in the loss function. This would yield a fully end-to-end training process, where the optimal value of $\mu$ is learned jointly with the network parameters. However, we found that an offline calibration of $\mu$, based on a simple scalar minimization of the entropy discrepancy prior to training, significantly reduces the computational burden and improves the stability of the optimization. This preliminary step decouples the tuning of the kinetic scale from the learning process, and results in faster and more robust convergence.
\end{remark}

\subsection{Space homogeneous Boltzmann equation with uncertain data}
All experiments in the homogeneous study were conducted using a single training instance randomly sampled from the data distribution, enabling the evaluation of the model's extrapolation capability to arbitrary unseen random points during testing.
Around 40000 available measurement data are randomly sampled and confined to the temporal subdomain $t \in [0, 2T/5]$.
The neural network architecture for both $f_B$ and $f_B^*$ consists of 7 fully-connected layers with 100 neurons each, using the hyperbolic tangent (tanh) activation function.

\subsubsection{Two bumps relaxation} \label{exam1}
In this example, we consider two bumps problem with uncertainty
\begin{equation*}
  f_0(z, v) = \frac{\rho_0}{2 \pi} \left( \exp{\left(- \frac{\vert v - s + d \vert^2}{\sigma} \right)} + \exp{\left(- \frac{\vert v - s - d\vert^2}{\sigma} \right)}\right),
\end{equation*}
where $d = 1.5$ and $s(z) = z_1(\sin{(2 \pi z_2)}, \cos{(2 \pi z_2)})^{\rm T}$ represents the stochastic displacement vector. 
Here, $z = (z_1, z_2)^{\rm T} \in (-1, 1) \times (0, 1)$ are random variables: $z_1$ controls the magnitude of displacement, and $z_2$ determines the direction angle. 
The parameter $\sigma = 0.5$ governs the thermal width of the bumps, and $\rho_0 = 0.75$ is the reference density. 
The velocity domain is truncated to $v \in [-10, 10] \times [-10, 10]$ with $\eps = 1$ and a final time $T = 2$. 
We note that this example constitutes a five-dimensional problem, comprising two dimensions in velocity, two in uncertainty, and one in time.
\begin{figure}[tb]
    \begin{center}
        \mbox{
        {\includegraphics[width = 0.4 \textwidth, trim=15 0 15 0,clip]{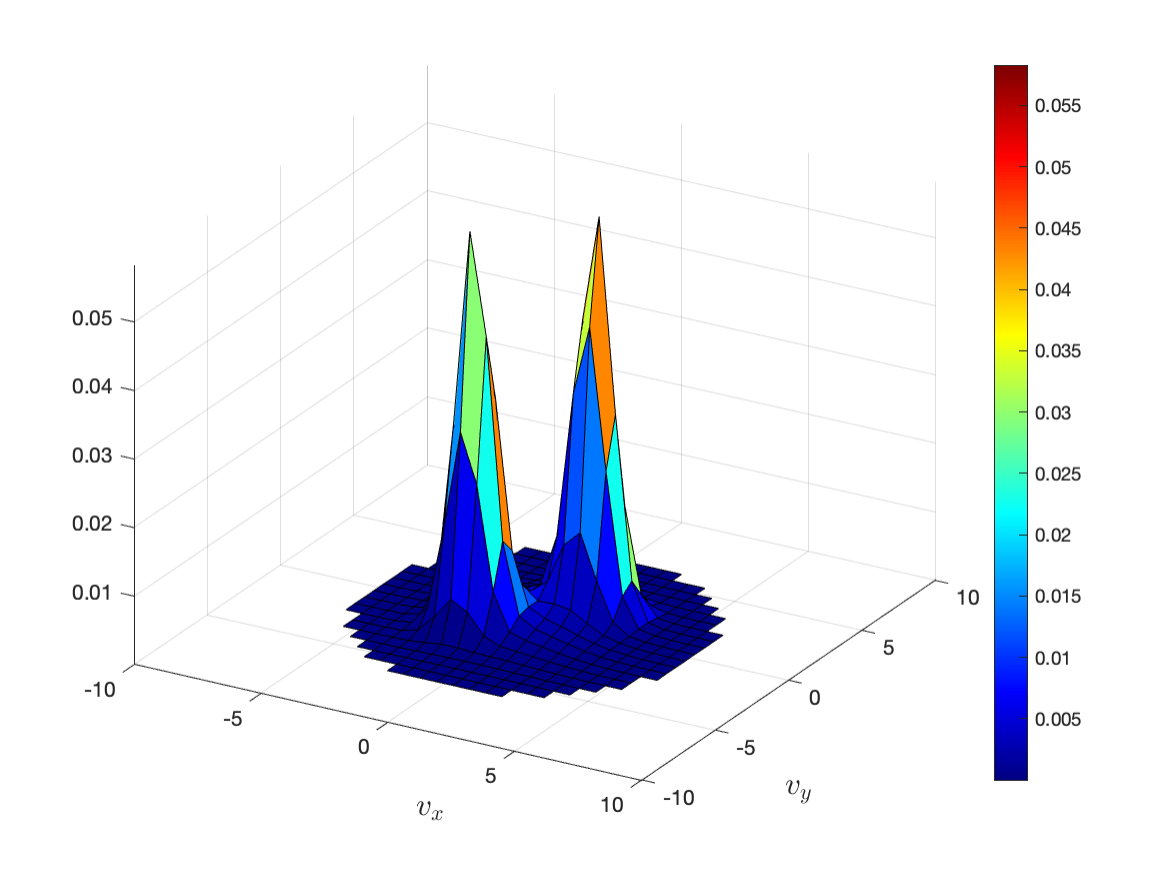}}
        {\includegraphics[width = 0.4 \textwidth, trim=15 0 15 0,clip]{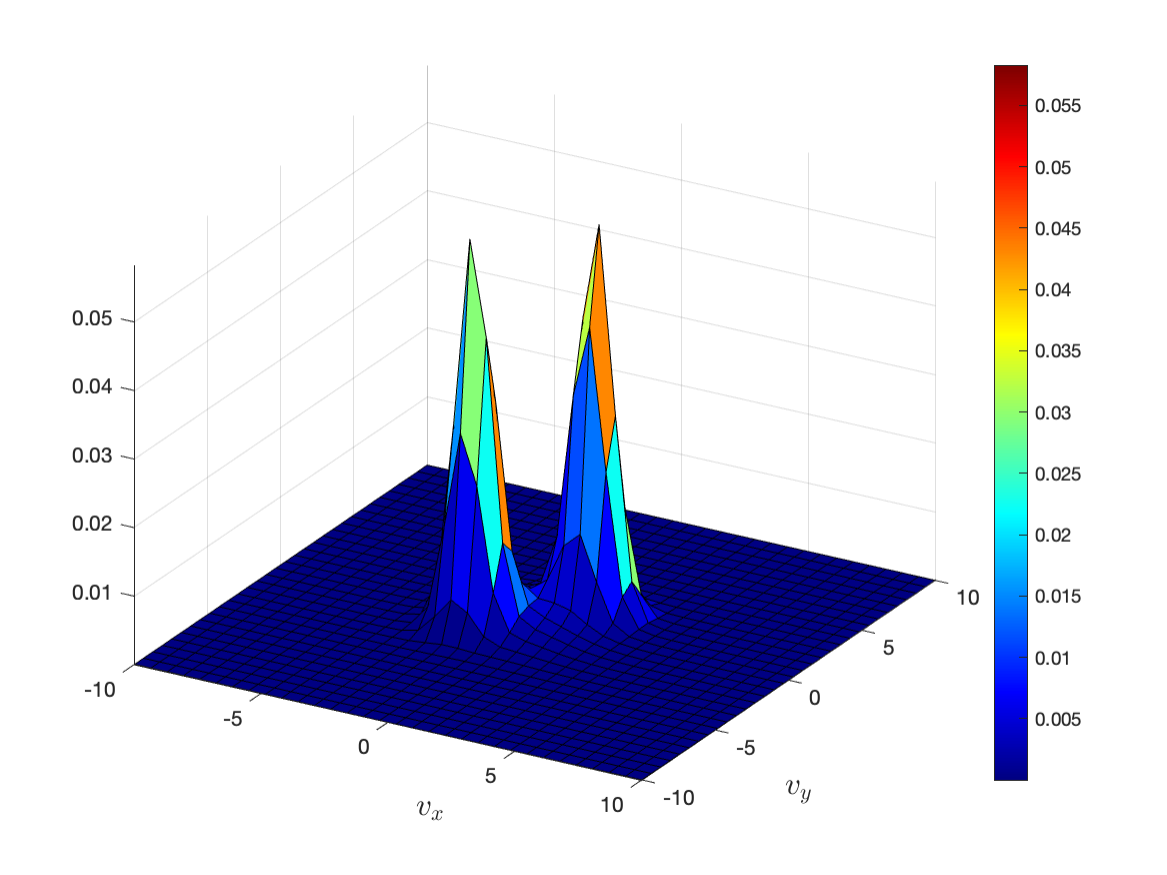}}}\\
        \mbox{
        {\includegraphics[width = 0.4 \textwidth, trim=15 0 15 0,clip]{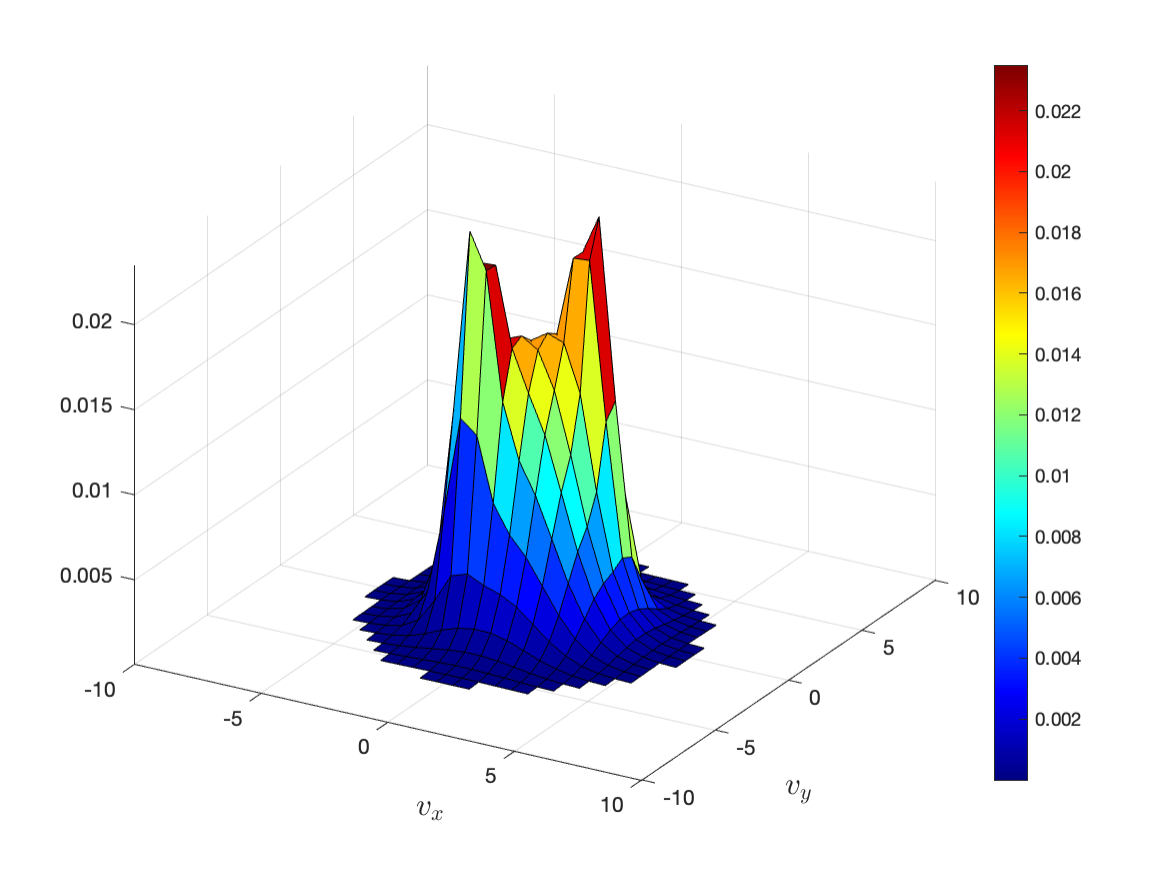}}
        {\includegraphics[width = 0.4 \textwidth, trim=15 0 15 0,clip]{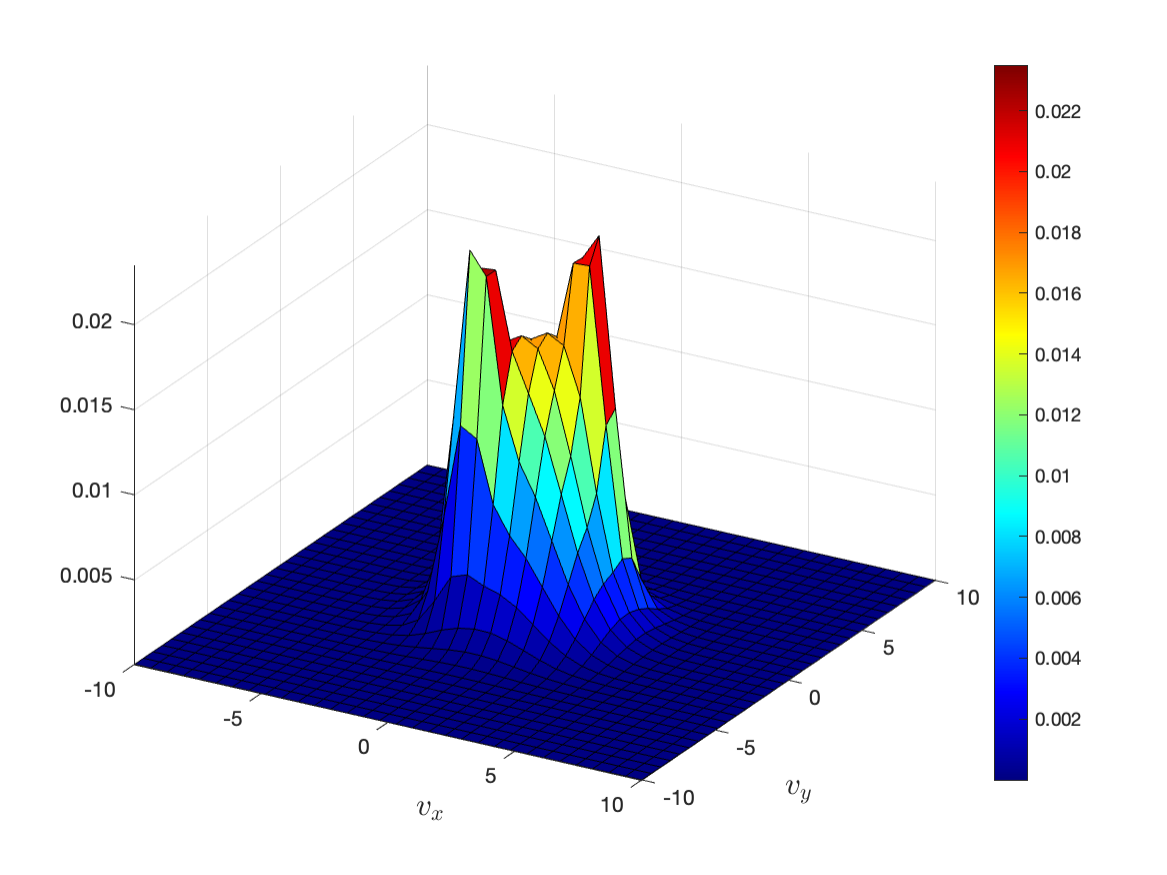}}}\\
        \caption{\sf $\mathbb{E}[f]$ for Example \ref{exam1}. Left: PINN; Right: SAPNN; Top: $t = 0.2$; Bottom: $t = 1.4$. The plots display only the positive part of the NN approximation.}
        \label{SP}
    \end{center}
\end{figure}
\begin{figure}[hbtp]
    \begin{center}
        \mbox{
        {\includegraphics[width = 0.4 \textwidth, trim=0 0 0 0,clip]{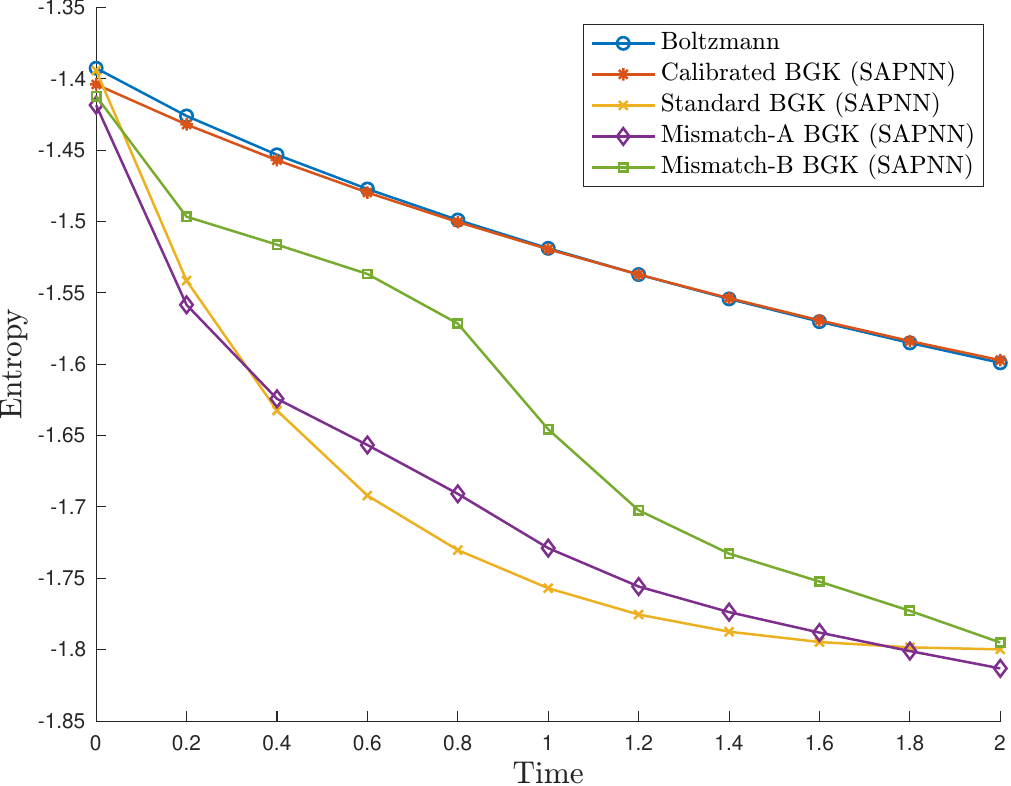}}}\\
        \caption{\sf Entropy for Example \ref{exam1}. Non-convex behaviors are observed unless the surrogate BGK model is aligned with the Boltzmann data.}
        \label{EN}
    \end{center}
\end{figure}

To demonstrate the advantages of structure preservation, we compare the expected values of $f$ obtained from PINN and SAPNN at different times $t = 0.2$ and $t = 1.4$, as shown in Fig.~\ref{SP}. 
While both methods effectively simulate the BGK model, the results produced by PINN -- without structure-preserving techniques -- exhibit unphysical negative values.

For this test case, minimization over the difference of the entropy functionals yields an optimal value $(\mu^*)^{-1}\approx  6.3$. Fig.~\ref{EN} presents the entropy evolution for the Boltzmann equation, the calibrated BGK (SAPNN), and the standard BGK (SAPNN), showing that the calibrated BGK entropy closely matches that of the Boltzmann model.

In addition, we designed two numerical schemes, the Mismatch-A BGK and Mismatch-B BGK, both trained using Boltzmann-derived data while maintaining a relaxation rate $\mu=1$. 
For Mismatch-A BGK, the loss weights align with those of the previously mentioned numerical schemes, whereas the PDE loss weight for Mismatch-B BGK is reduced to $1/20$ of the original value.
The numerical results for both schemes are presented in Fig.~\ref{EN}. 
Due to significant incompatibility between the Boltzmann training data and the BGK model with $\eps=1$, both entropy profiles exhibit poor performance. 
In Mismatch-B BGK, where data loss dominates, the entropy initially aligns more closely with the Boltzmann model in regions supervised by training data. 
However, as time progresses and data supervision diminishes, the entropy rapidly decays toward the Standard BGK profile. 
Conversely, Mismatch-A BGK—with its higher PDE loss weight—is dominated by PDE model errors, causing its entropy to closely follow the Standard BGK profile. 
Nevertheless, incompatibility with training data introduces oscillations in the entropy. 

These counterexamples highlights two critical points: 
\begin{itemize}
\item surrogate models in PINN must be highly compatible with the training data;
\item the model selection in PINN often has a greater influence on performance than the training data itself.
\end{itemize}

Fig.~\ref{EM} presents the $L_1$ error of expected values $\mathbb{E}[f]$ comparisons between the Boltzmann model and two BGK-based models:  the SAPNN-based standard BGK model $\left( {\rm i.e.} \; \Vert \mathbb{E}[f_B] - \mathbb{E}[f_{Bolt}] \Vert_1 \right)$, and the APNN-based calibrated BGK model (i.e. $\Vert \mathbb{E}[f_B^*] - \mathbb{E}[f_{Bolt}] \Vert_1$) at different temporal stages. 
Clearly, the $L_1$ error further confirms that the calibrated BGK model provides a closer approximation to the Boltzmann model than the standard BGK model.
The standard BGK model shares the same initial condition with the Boltzmann model (yielding zero error at $t = 0$); however, due to inherent differences between the two models, the error gradually increases over time.
With $\eps = 1$, the standard BGK model achieves equilibrium state faster than the Boltzmann solution (reaching equilibrium at $t\approx2$ per entropy analysis, see Fig.~\ref{EN}), resulting in peak errors near this temporal boundary. 
For $t>2$, as the Boltzmann solution continues evolving toward the same equilibrium state already attained by the standard BGK, the errors decrease rapidly.
While showing similar error growth-decay patterns, the calibrated BGK demonstrates: milder error growth and slower error decay post-peak due to synchronized approach to equilibrium state.
This improved behavior stems from the calibrated BGK's enhanced fidelity to the Boltzmann dynamics through velocity-dependent collision parameters.

In Fig.~\ref{SLT}, we show the $L_2$ relative error for short-time behavior ($T = 2$) and long-time behavior ($T = 20$) of different methods with $K = 50$ samples for expected value and $L = 16300$ samples for control variate.
For this problem, initial perturbations induce significant variation across samples at the onset. 
As time evolves, all samples converge toward their respective equilibrium, which exhibit minimal inter-state discrepancies. 
Consequently, the MC method rapidly reduces uncertainty errors during this transient phase, followed by stabilization as equilibrium is approached.
Compared to the MC method, both SAPNN($f_B$) and SAPNN($f_B^*$) can significantly reduce the error. 
Since $f_B^*$ is derived from the calibrated BGK model, which more closely approximates the Boltzmann dynamics than the standard BGK model, SAPNN($f_B^*$) notably outperforms  SAPNN($f_B$).
We also observe that, due to the inherent discrepancy between the BGK and Boltzmann models, the error curves for SAPNN($f_B$) and SAPNN($f_B^*$) closely resemble those of $\Vert \mathbb{E}[f_B] - \mathbb{E}[f_{Bolt}] \Vert_1$ and $\Vert \mathbb{E}[f_B^*] - \mathbb{E}[f_{Bolt}] \Vert_1$ in Fig.~\ref{EM}, respectively.
Furthermore, the SAPNN($f_B$, $f_B^*$) approach, based on the multiple variance reduction Monte Carlo method, achieves the lowest error among all methods compared.
\begin{figure}[tb]
    \begin{center}
        \mbox{
        {\includegraphics[width = 0.4 \textwidth, trim=-18 0 0 -18,clip]{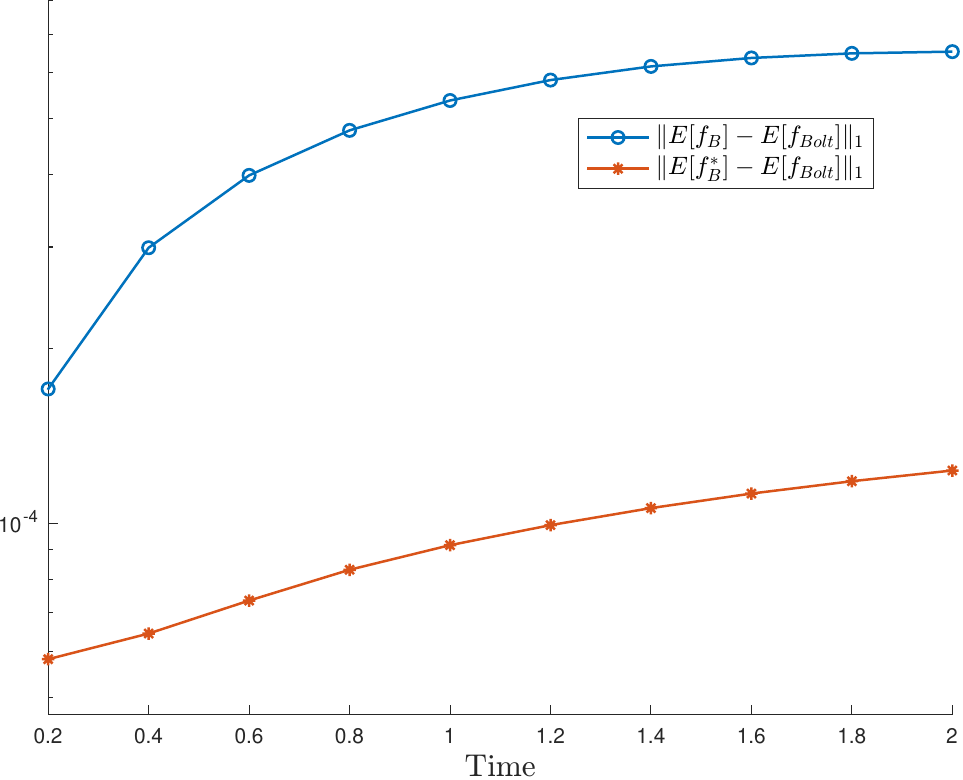}}
        {\includegraphics[width = 0.4 \textwidth, trim=-18 0 0 -18,clip]{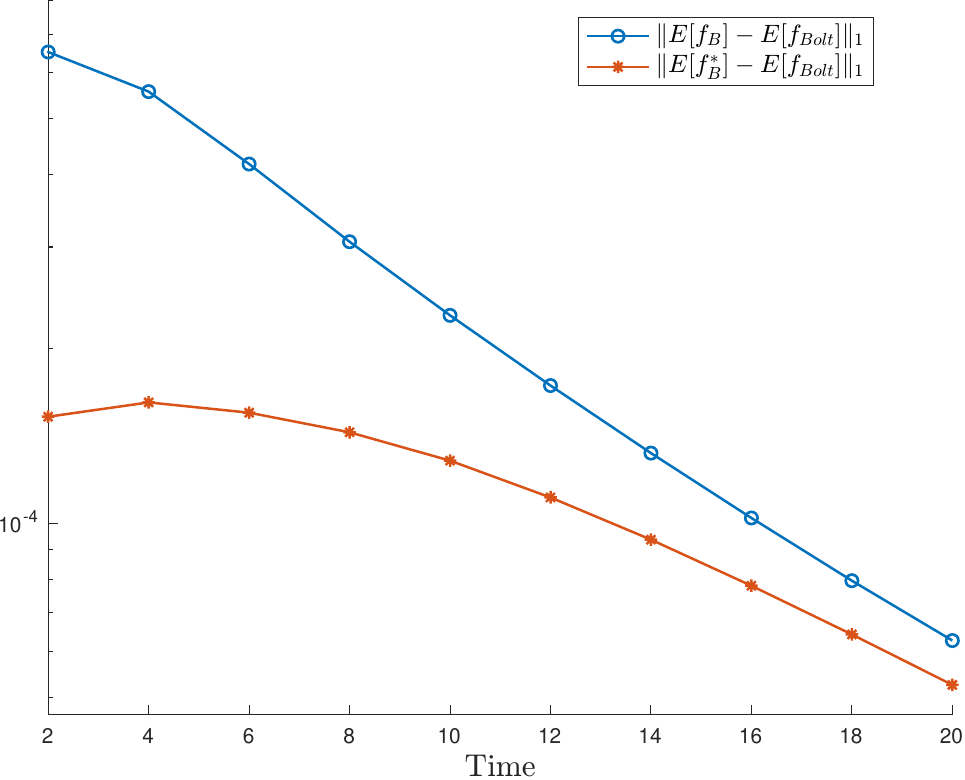}}
        }\\
        \caption{\sf Error of model for Example \ref{exam1}. Left: Short-time behavior; Right: Long-time behavior.}
        \label{EM}
    \end{center}
\end{figure}

\begin{figure}[tb]
    \begin{center}
        \mbox{
        {\includegraphics[width = 0.4 \textwidth, trim=0 0 0 0,clip]{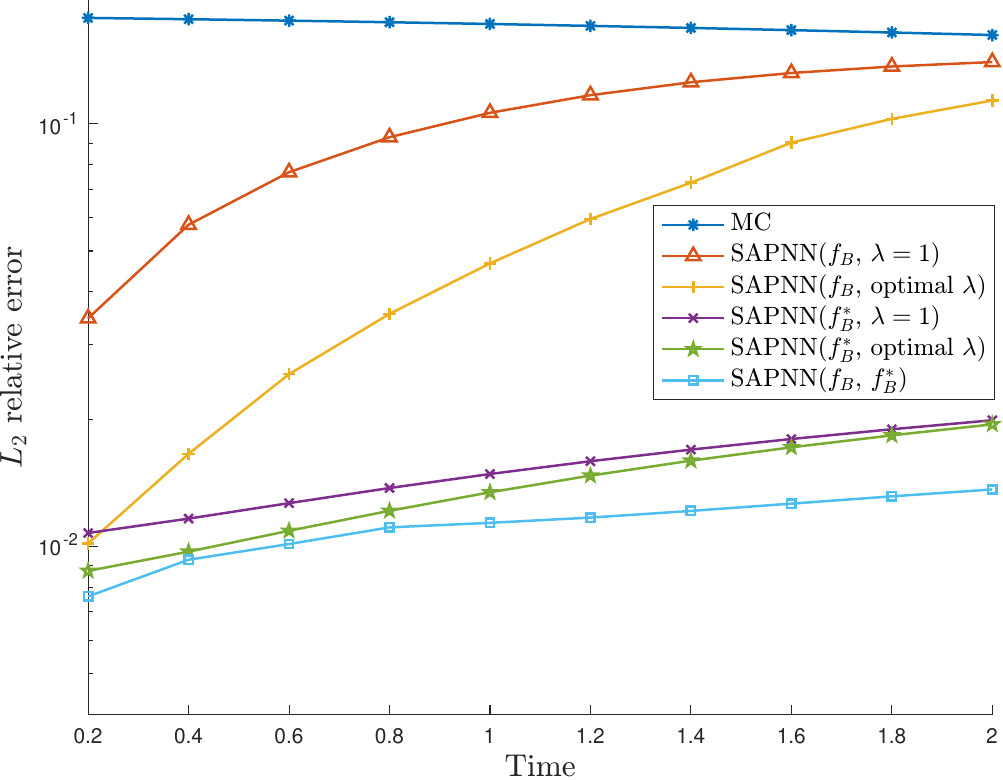}}
        {\includegraphics[width = 0.4 \textwidth, trim=0 0 0 0,clip]{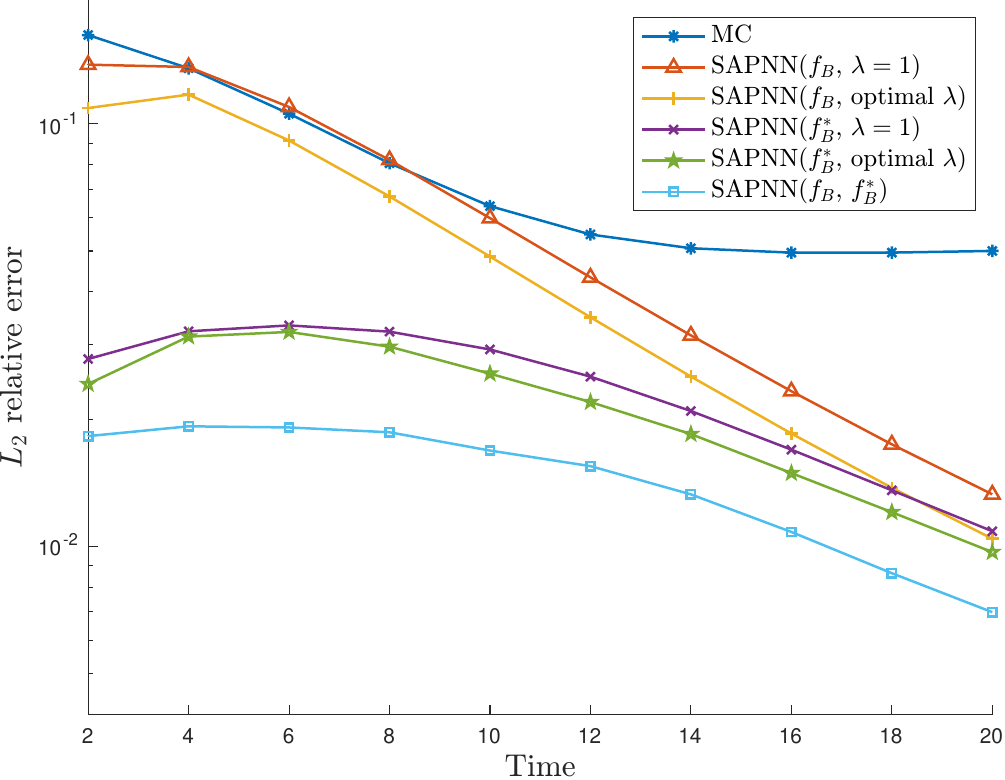}}
        }\\
        \caption{\sf $L_2$ relative errors of $\mathbb{E}[f]$ for Example \ref{exam1}. The number of samples used to compute the expected value and to construct the control variate are $K = 50$ and $L = 16300$, respectively. Left: Short-time behavior; Right: Long-time behavior.}
        \label{SLT}
    \end{center}
\end{figure}

\begin{remark}
  In the two-bubble case study, where the parameters $\rho_0$, $\sigma$, and $d$ are fixed, the optimal BGK relaxation rate $\mu^*$ remains invariant. More generally, $\mu^*$ is expected to depend intrinsically on these parameters, as illustrated in Fig.~\ref{mu}. 
  Interestingly, in the context of space non-homogeneous problems, the method exhibits a significantly reduced sensitivity to the precise value of $\mu$. Our numerical experiments show that using the calibrated value $\mu^* = (6.3)^{-1}$ consistently yields satisfactory results. Further fine-tuning did not produce noticeable improvements and introduced additional computational cost.
This insensitivity likely stems from the dominant role played by transport dynamics and the stabilizing influence of spatial coupling, which mitigate the effect of local relaxation discrepancies. 
\end{remark}

\begin{figure}[hbt]
    \begin{center}
        \mbox{
        {\includegraphics[width = 0.4 \textwidth, trim=35 200 70 210,clip]{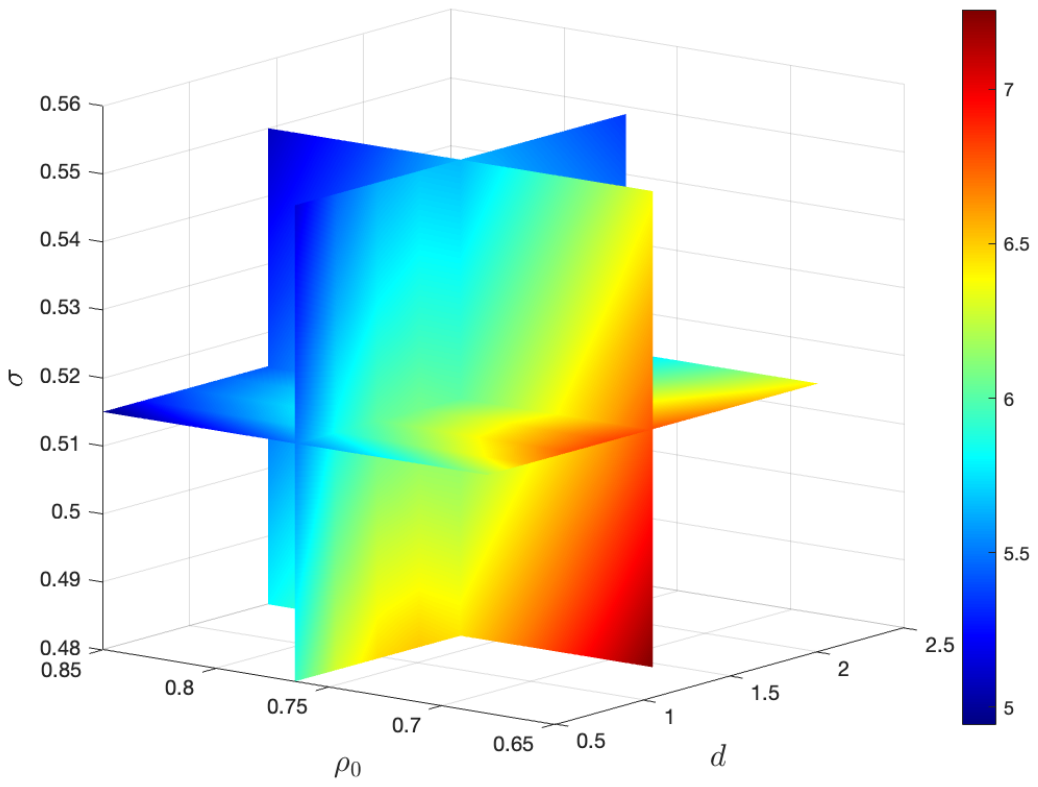}}
        }\\
        \caption{\sf Sensitivity of the optimal relaxation rate $(\mu^*)^{-1}$ to the initial data parameters.  }
        \label{mu}
    \end{center}
\end{figure}

\begin{figure}[htb]
    \begin{center}
        \mbox{
        {\includegraphics[width = 0.4 \textwidth, trim=0 0 0 0,clip]{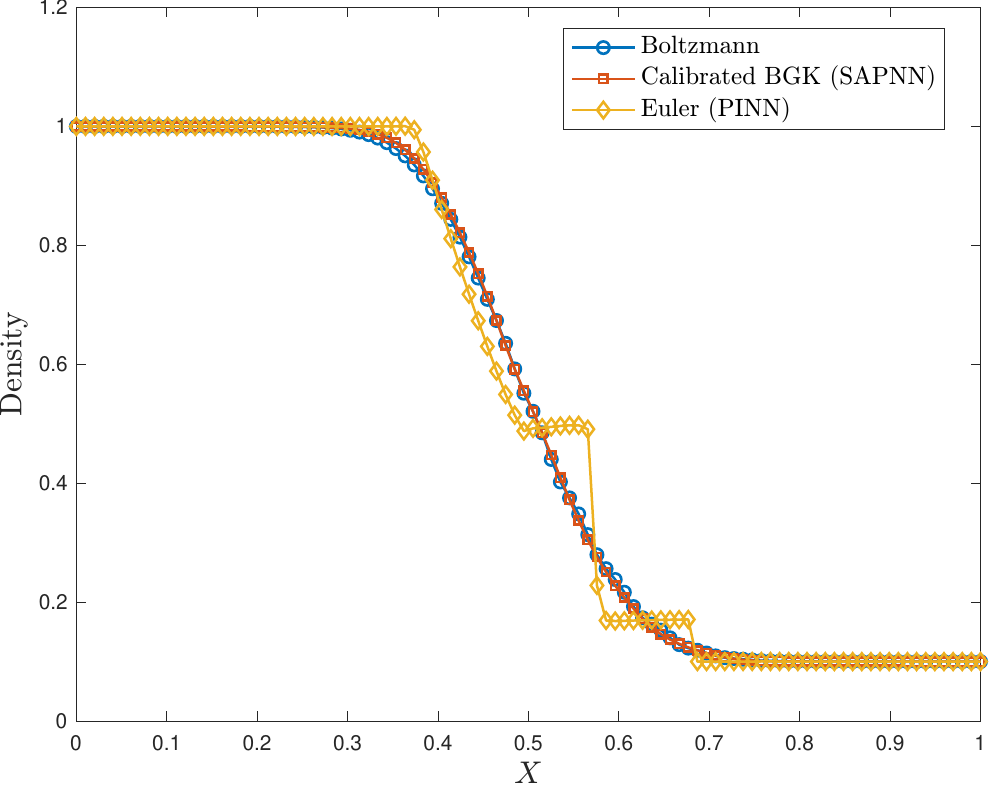}}
        {\includegraphics[width = 0.4 \textwidth, trim=0 0 0 0,clip]{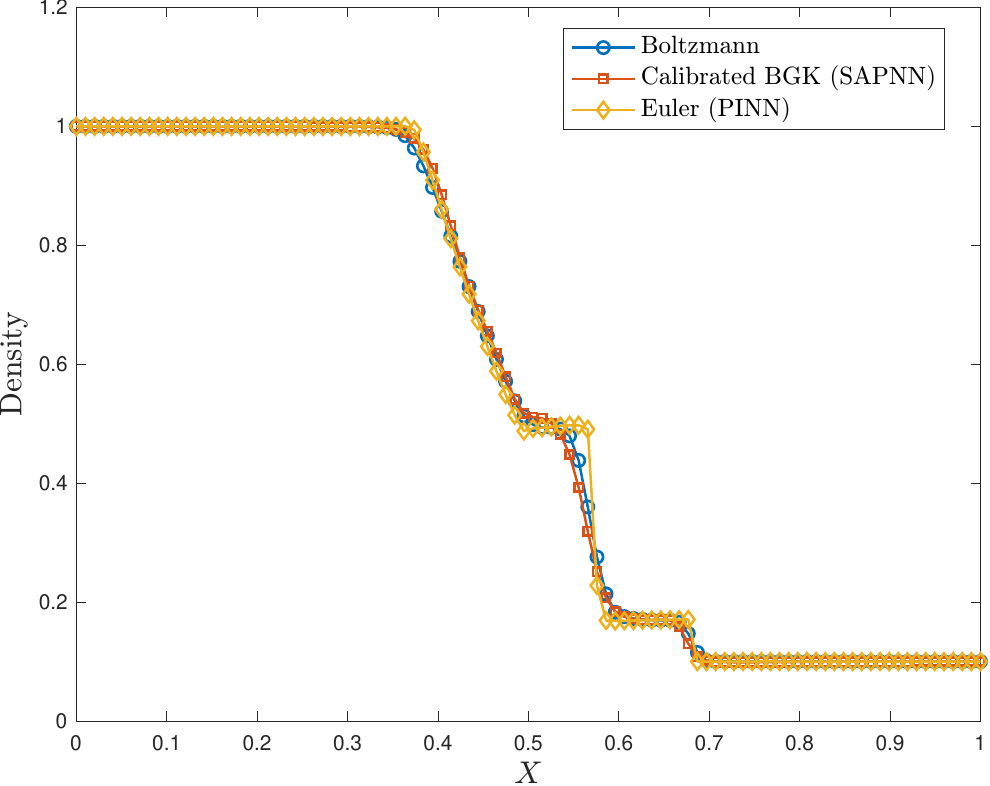}}}\\
        \mbox{
        {\includegraphics[width = 0.4 \textwidth, trim=0 0 0 -20,clip]{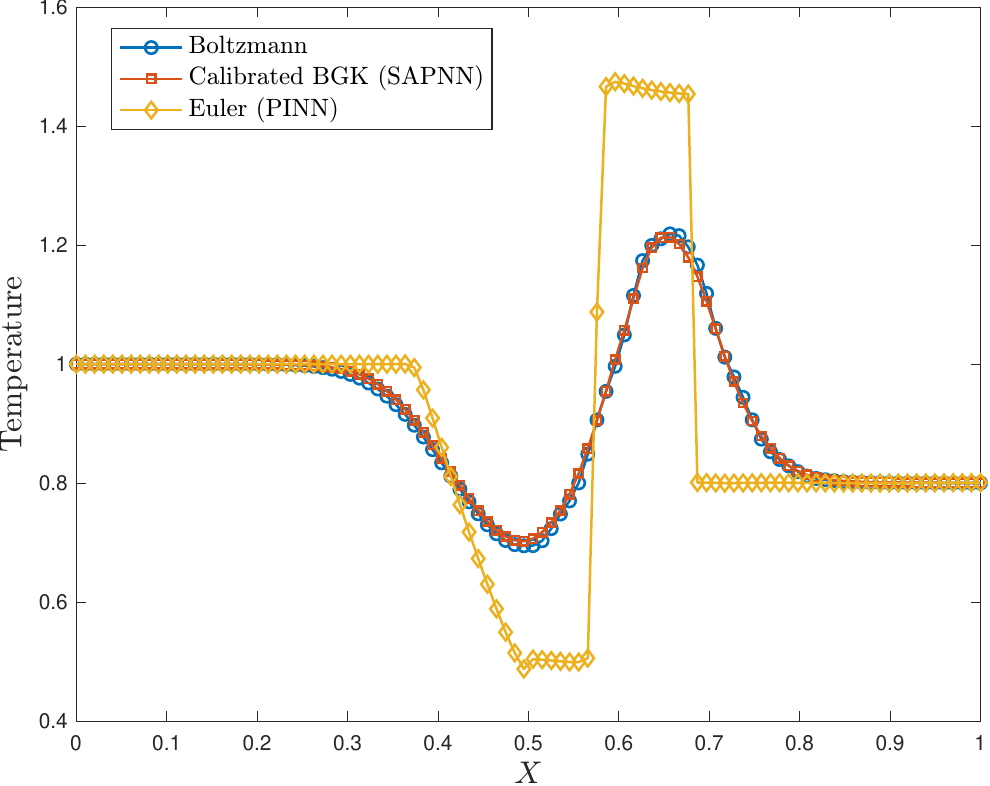}}
        {\includegraphics[width = 0.4 \textwidth, trim=0 0 0 -20,clip]{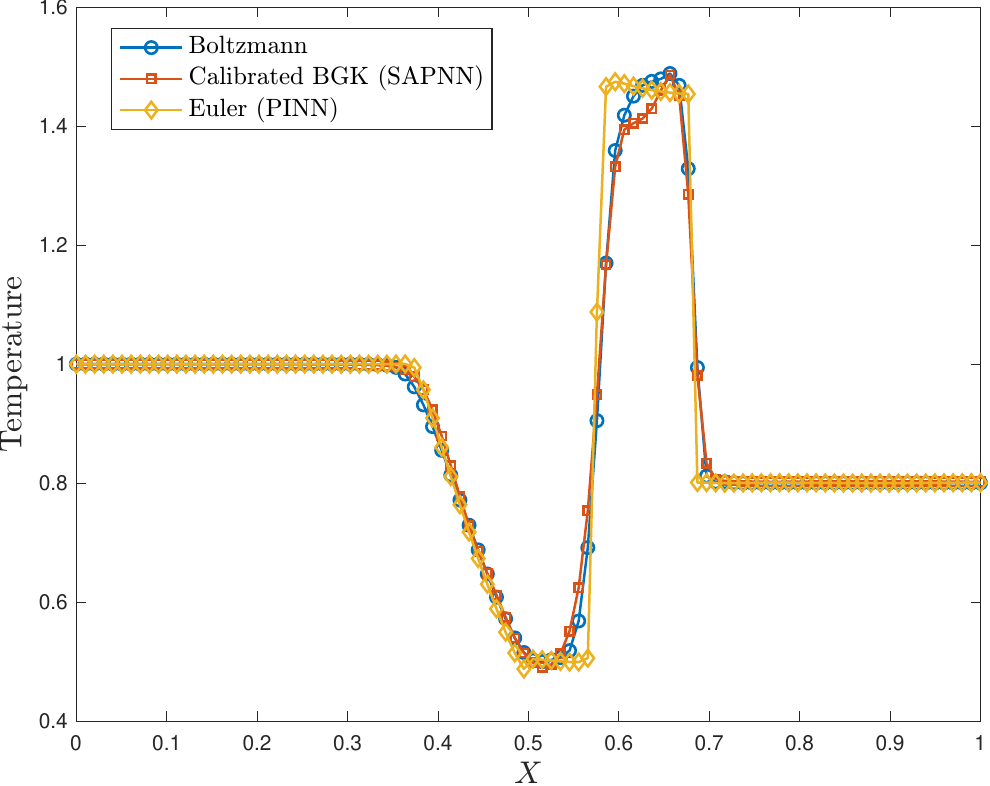}}}\\
        \caption{\sf Expectation for Example \ref{exam3}. Left: $\eps = 10^{-2}$; Right: $\eps = 2 \times 10^{-4}$; Top: Density; Bottom: Temperature.}
        \label{Exp3}
    \end{center}
\end{figure}

\subsection{Space non-homogeneous Boltzmann equation with uncertain data}
In the subsequent examples, the available measurement data are randomly sampled and confined to the temporal subdomain $t \in [0, 3T/5]$. 
All Euler PINN models are trained using $60$ samples with a fully connected neural network of architecture $100 \times 8$. 
The BGK SAPNN models are trained on $35$ samples and comprises two components: the first employs a $150 \times 8$ fully connected network to approximate the BGK equation, while the second uses a $100 \times 8$ network to model the evolution of macroscopic quantities.

\subsubsection{Sod problem} \label{exam3}
In this example, we will consider the Sod test with the following uncertain initial data
\begin{equation*}
\begin{aligned}
  &\rho_0(x) = 1, \quad &&T_0(z, x) = 1 + sz \quad &&{\rm if}\; 0 \leq x \leq 0.5 \\
  &\rho_0(x) = 0.125, \quad &&T_0(z, x) = 0.8 + sz \quad &&{\rm if}\; 0.5 < x \leq 1 \\
\end{aligned}
\end{equation*}
with $s = 0.25$, $z$ uniform in $[-1, 1]$ and equilibrium initial distribution
\begin{equation*}
  f_0(z, x, v) = \frac{\rho_0(x)}{2 \pi} \exp{\left(-\frac{\vert v \vert^2}{2 T_0(z,x)}\right)}.
\end{equation*} 
The Sod problem is one of the most extensively studied Riemann problems and features three fundamental wave structures: a shock wave, a contact discontinuity, and a rarefaction wave.
The truncation of the velocity space as well the other numerical parameters are the same as in Example \ref{exam2}.
We consider the cases $\eps = 10^{-2}$ and $\eps = 2 \times 10^{-4}$ for the Boltzmann equation, and a calibrated relaxation parameter $(\mu^*)^{-1} \approx 6.3$ for the BGK model with some random Boltzmann data for training.
With the calibrated parameter $\mu^*$, the BGK solution aligns more closely with that of the Boltzmann equation.
The final simulation time is set to $T = 0.0875$. 

\begin{figure}[htb]
    \begin{center}
        \mbox{
        {\includegraphics[width = 0.4 \textwidth, trim=0 0 0 0,clip]{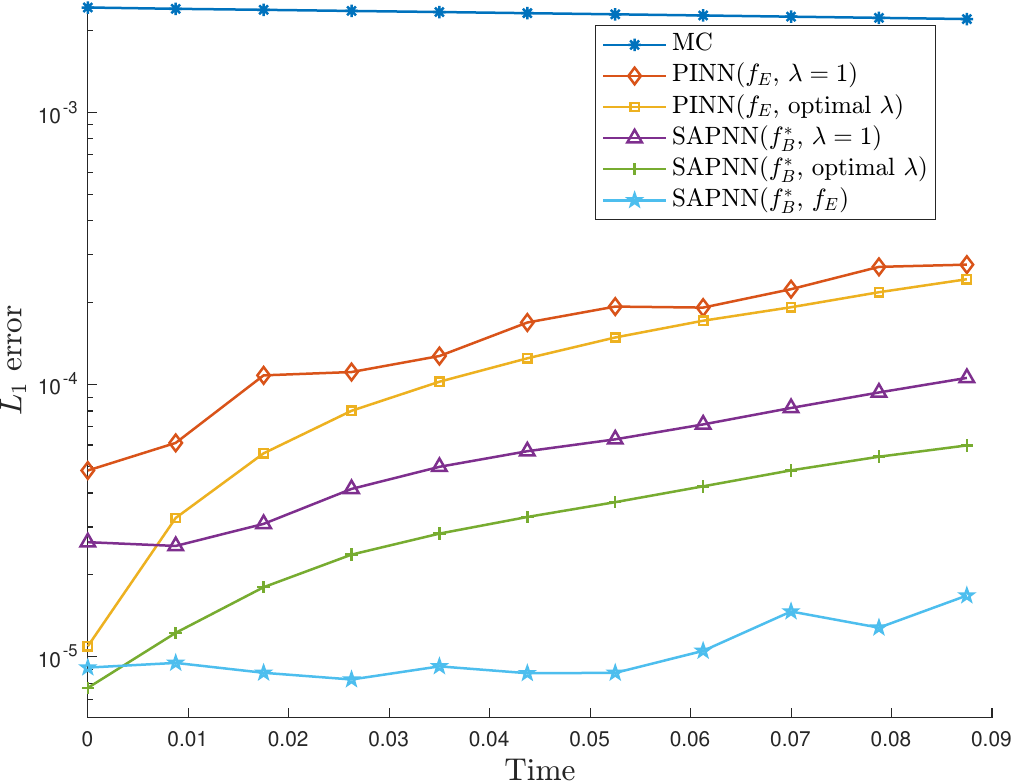}}
        {\includegraphics[width = 0.4 \textwidth, trim=0 0 0 0,clip]{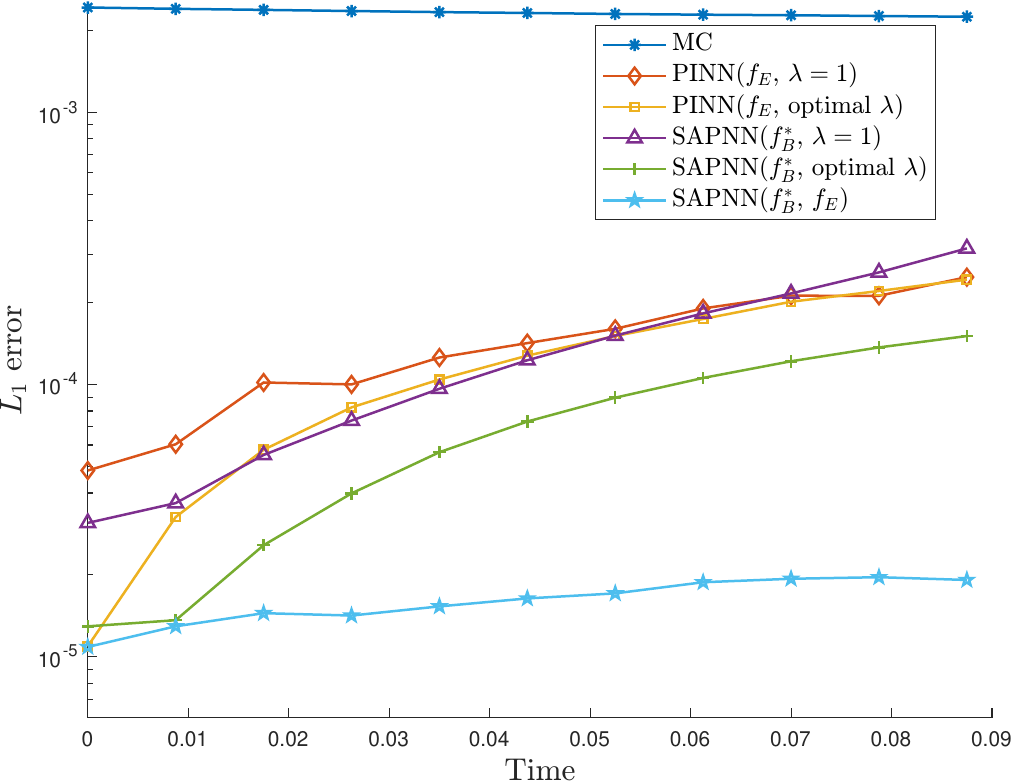}}}\\
        \mbox{
        {\includegraphics[width = 0.4 \textwidth, trim=0 0 0 -20,clip]{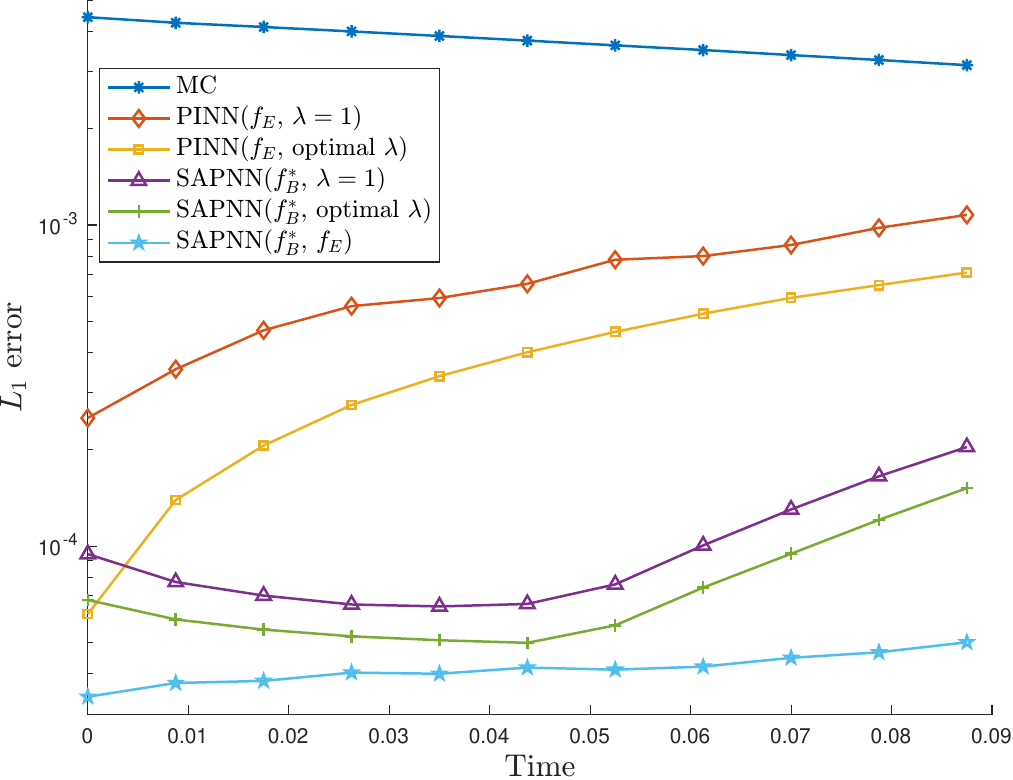}}
        {\includegraphics[width = 0.4 \textwidth, trim=0 0 0 -20,clip]{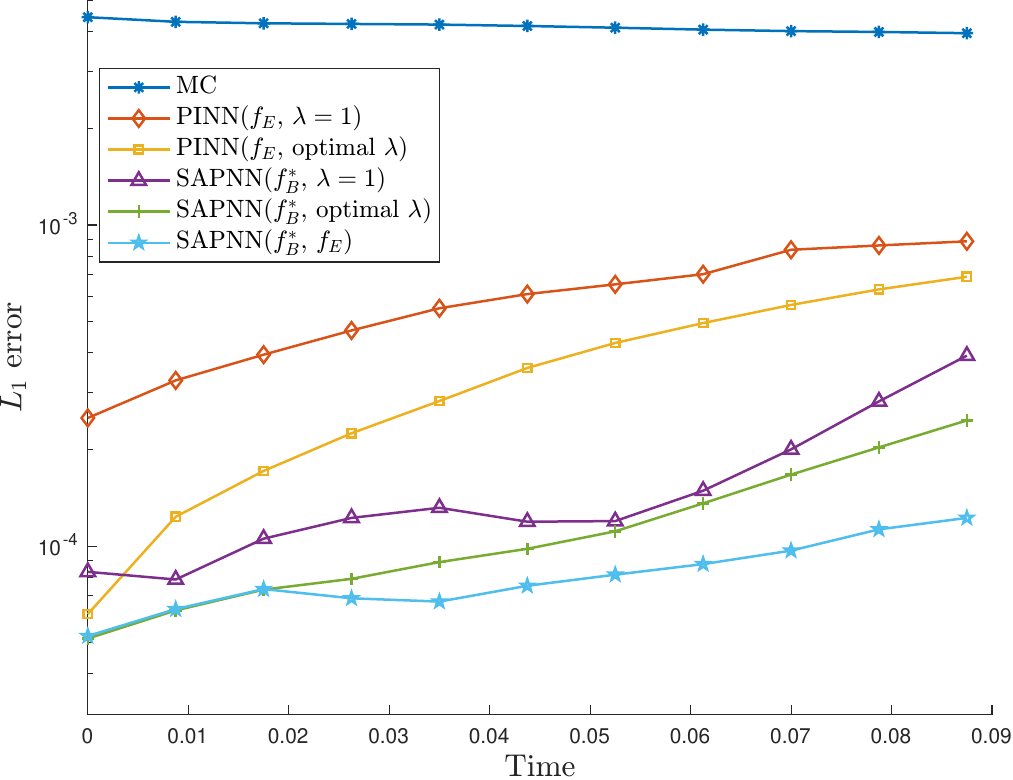}}}\\
        \caption{\sf $L_1$ error of expectation for Example \ref{exam3}. The number of samples used to compute the expected value and to construct the control variate are $K = 30$ and $L = 5000$, respectively. Left: $\eps = 10^{-2}$; Right: $\eps = 2 \times 10^{-4}$; Top: Density; Bottom: Temperature.}
        \label{SodError}
    \end{center}
\end{figure}

Figure~\ref{Exp3} presents the expected density and temperature profiles at the final time. 
We observe that the calibrated BGK model closely approximates the Boltzmann solution, and both models tend toward the Euler limit as the Knudsen number $\eps$ decreases.
Fig.~\ref{SodError} presents the $L_{1}$ errors in the expected density and temperature computed using MC, PINN($f_E$), SAPNN($f_B^*$), and SAPNN($f_B^*, f_E$).
The number of samples used to compute the expected value of the solution is $K = 30$ while the number of samples used to compute the control variate is $L = 5000$.
As anticipated, SAPNN($f_B^*$) yields a more substantial error reduction than PINN($f_E$), while SAPNN($f_B^*, f_E$)  method yields the lowest error among all methods considered.

\subsubsection{Lax problem} \label{exam4}
We consider the Lax problem with uncertainty in both density and temperature:
\begin{equation*}
  (\rho, u_x, u_y, T)(x, t = 0) = \left \{
  \begin{aligned}
    &(0.445 + 0.02z_1,\; 0.698,\; 0,\; 3.528) \quad &&{\rm if} \; 0 \leq x \leq 0.5 \\
    &(0.5,\; 0,\; 0,\; 0.571 + 0.02z_2) \quad &&{\rm if} \; 0.5 < x \leq 1
  \end{aligned}
  \right.
\end{equation*}
where the equilibrium initial distribution is given by
\begin{equation*}
  f_0(z, x, v) = \frac{\rho_0(z, x)}{2 \pi T_0(z, x)} \exp{\left(-\frac{\vert v - u_0(x)\vert^2}{2 T_0(z, x)}\right)}.
\end{equation*}
The initial condition of the Lax problem is similar to that of the Sod problem, as it also involves a compression wave and a rarefaction wave. However, the stronger contrast between the high-density compression region and the low-density rarefaction region in the Lax problem results in more intense wave propagation.
The velocity space is truncated to the domain $[-15, 15] \times [-15, 15]$, and the final time is set to $0.04375$.
We employ relaxation parameters $\eps = 10^{-2}$ and $2 \times 10^{-4}$ for the Boltzmann equation, while for the calibrated BGK model we set $(\mu^*)^{-1} = 6.3$, with some random Boltzmann data for training.

\begin{figure}[htb]
    \begin{center}
        \mbox{
        {\includegraphics[width = 0.4 \textwidth, trim=0 0 0 0,clip]{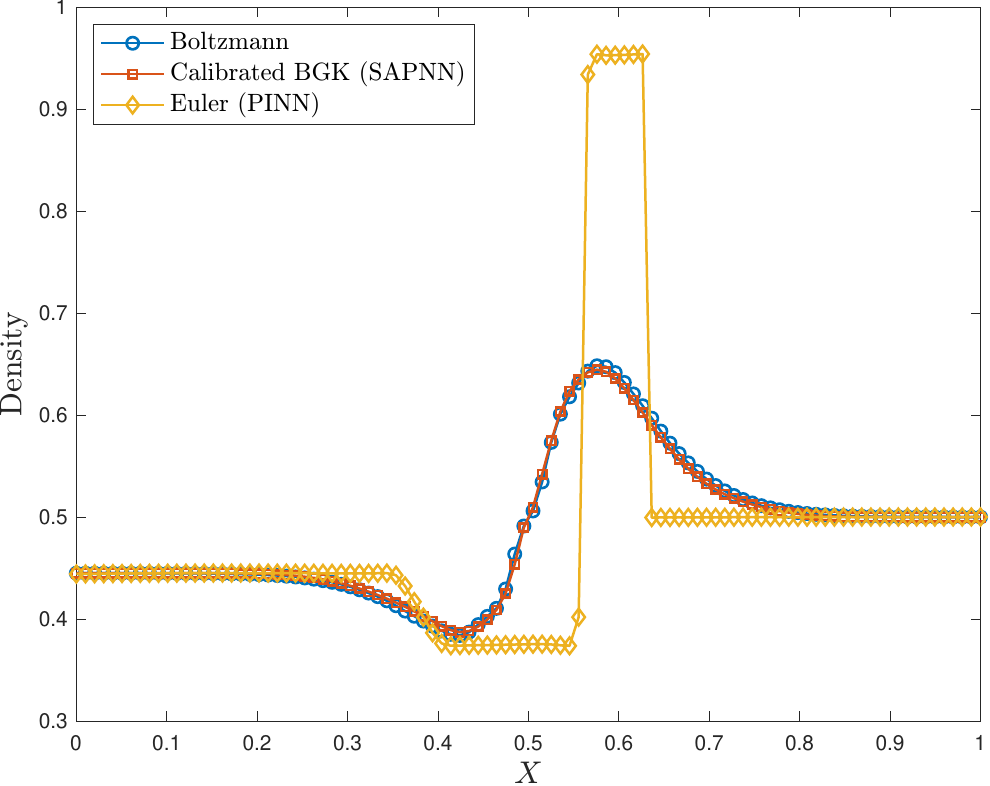}}
        {\includegraphics[width = 0.4 \textwidth, trim=0 0 0 0,clip]{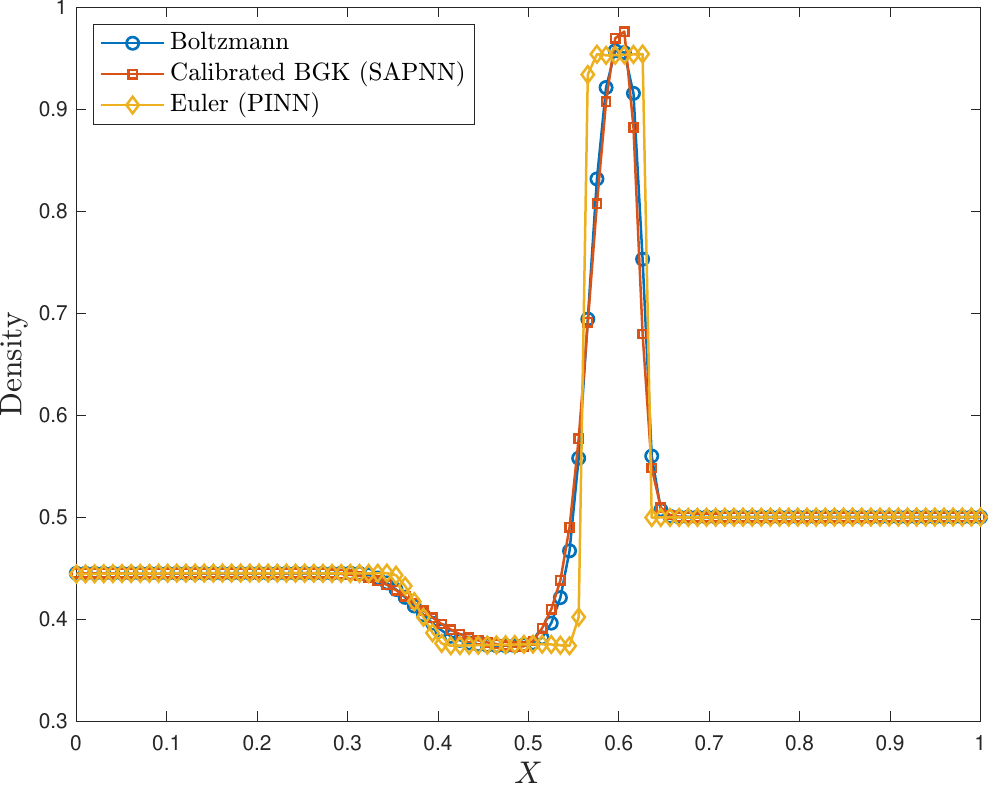}}}\\
        \mbox{
        {\includegraphics[width = 0.4 \textwidth, trim=0 0 0 -20,clip]{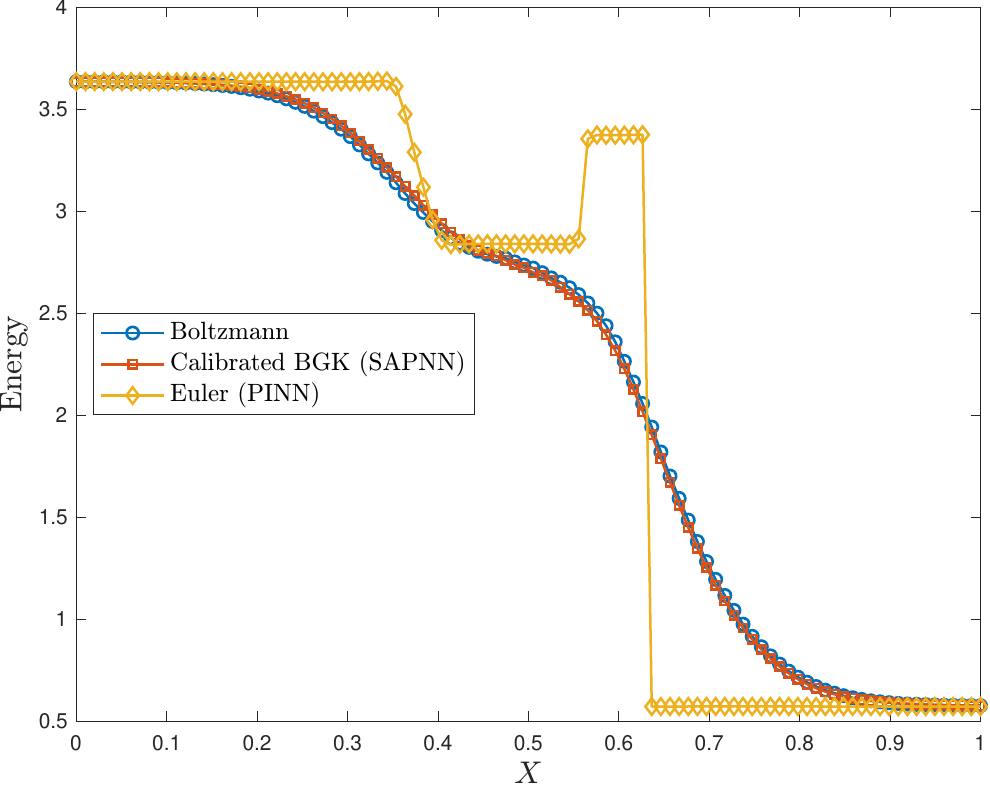}}
        {\includegraphics[width = 0.4 \textwidth, trim=0 0 0 -20,clip]{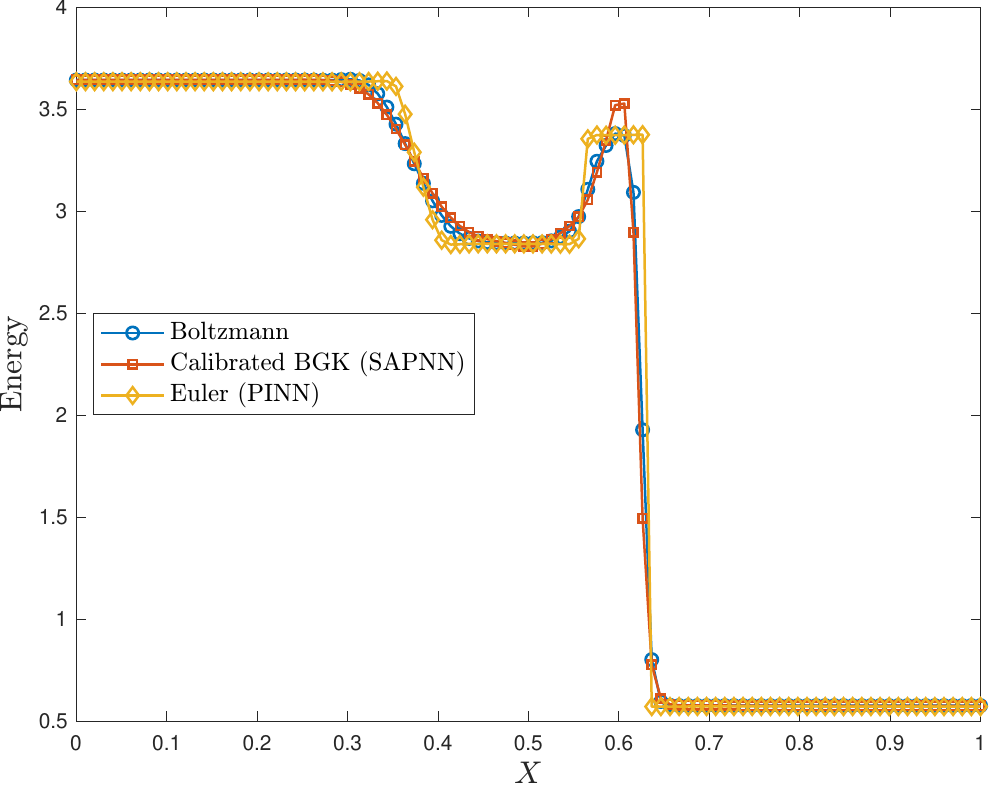}}}\\
        \caption{\sf Expectation for Example \ref{exam4}. Left: $\eps = 10^{-2}$; Right: $\eps = 2 \times 10^{-4}$; Top: Density; Bottom: Energy.}
        \label{Exp4}
    \end{center}
\end{figure}

We set the number of samples used to compute the expected value of the solution to $K = 30$, and the number of samples used for evaluating the control variate to $L = 5000$.
Figure~\ref{Exp4} displays the expected density and energy profiles at different Knudsen numbers, demonstrating that the SAPNN solvers for $f_B^*$ at $\eps = 2 \times 10^{-4}$ and $f_E$ accurately capture shock waves.
As shown in Fig.~\ref{LaxError}, SAPNN($f_B^*$) once again outperforms PINN($f_E$), and the additional performance gain provided by the combined SAPNN($f_B^*, f_E$) method is clearly evident.

\subsubsection{Double rarefaction problem}  \label{exam2} We consider an equilibrium initial distribution for the double rarefaction problem:
\begin{equation*}
  f_0(z, x, v) = \frac{\rho_0(x)}{2 \pi T_0(x)} \exp{\left(-\frac{\vert v - u_0(z, x)\vert^2}{2 T_0(x)}\right)}
\end{equation*} 
where
\begin{equation*}
  (\rho, u_x, u_y, T)(x, t = 0) = \left \{
  \begin{aligned}
    &(1,\; -2 + 0.05 z_1,\; 0,\; 0.4) \quad &&{\rm if} \; 0 \leq x \leq 0.5 \\
    &(1,\; 2 + 0.05 z_2,\; 0,\; 0.4) \quad &&{\rm if} \; 0.5 < x \leq 1
  \end{aligned}
  \right.
\end{equation*}
and the uncertainty is introduced through the velocity component $u_x$.

\begin{figure}[htb]
    \begin{center}
        \mbox{
        {\includegraphics[width = 0.4 \textwidth, trim=0 0 0 0,clip]{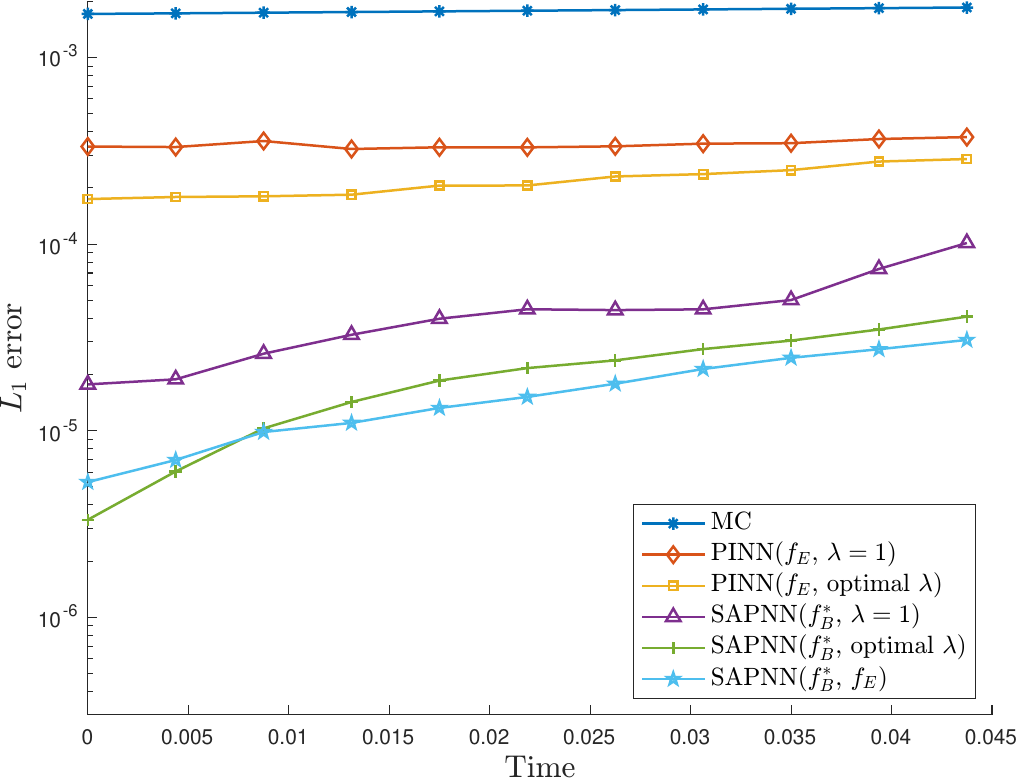}}
        {\includegraphics[width = 0.4 \textwidth, trim=0 0 0 0,clip]{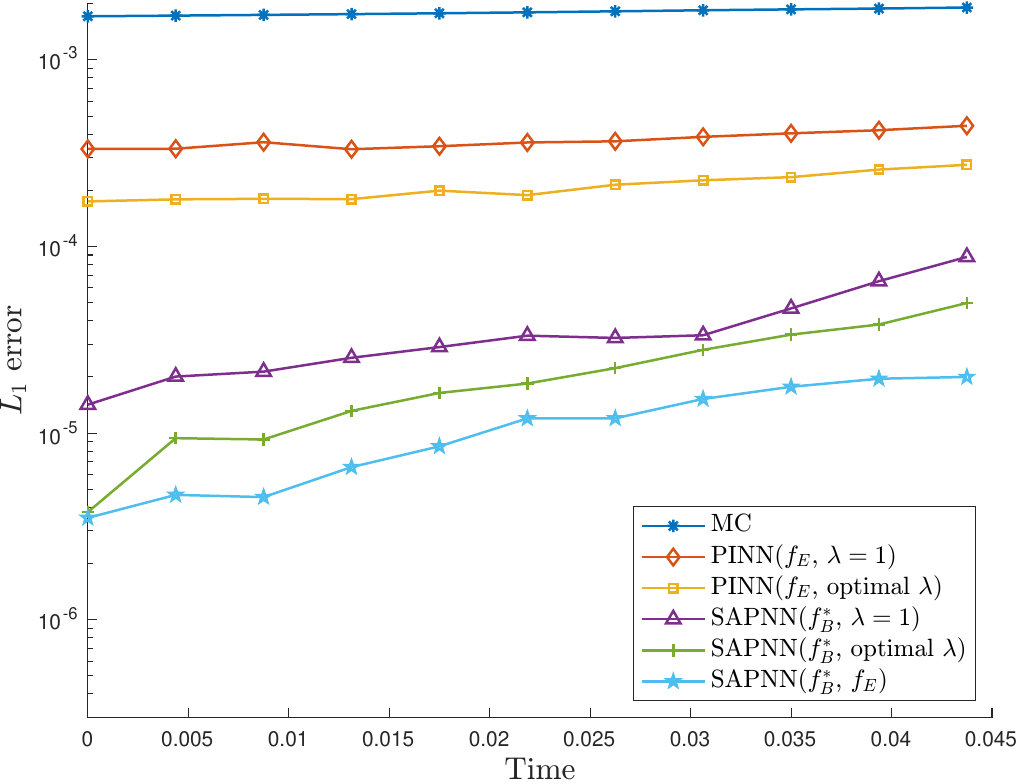}}}\\
        \mbox{
        {\includegraphics[width = 0.4 \textwidth, trim=0 0 0 -20,clip]{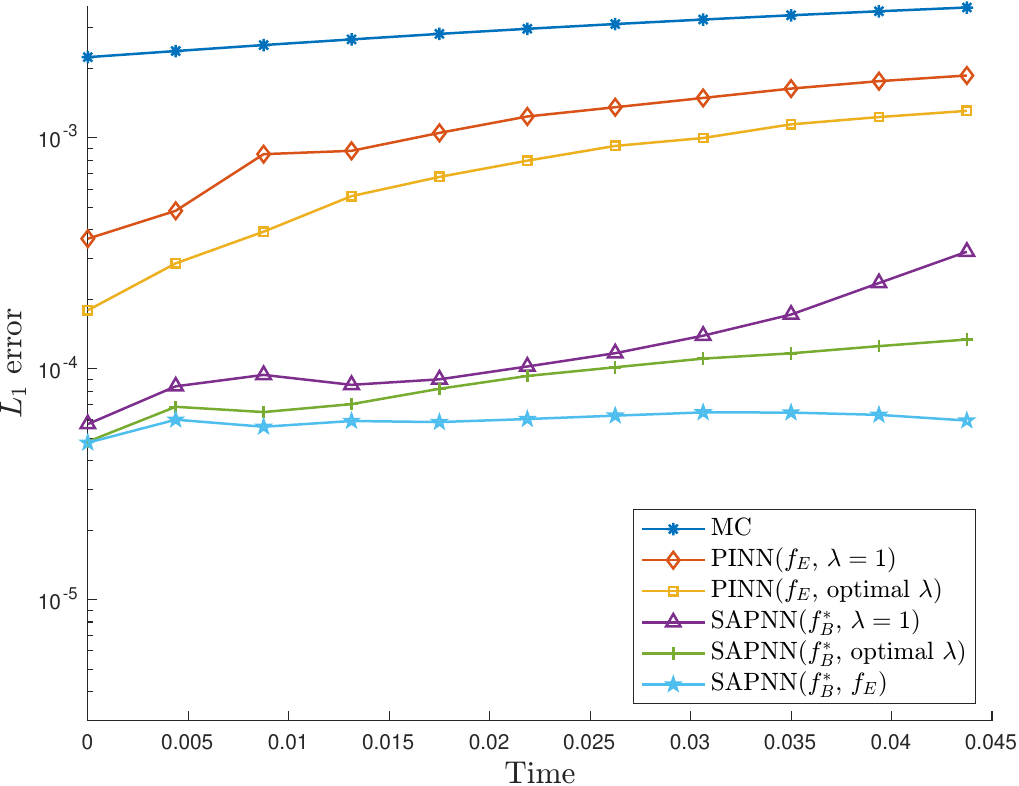}}
        {\includegraphics[width = 0.4 \textwidth, trim=0 0 0 -20,clip]{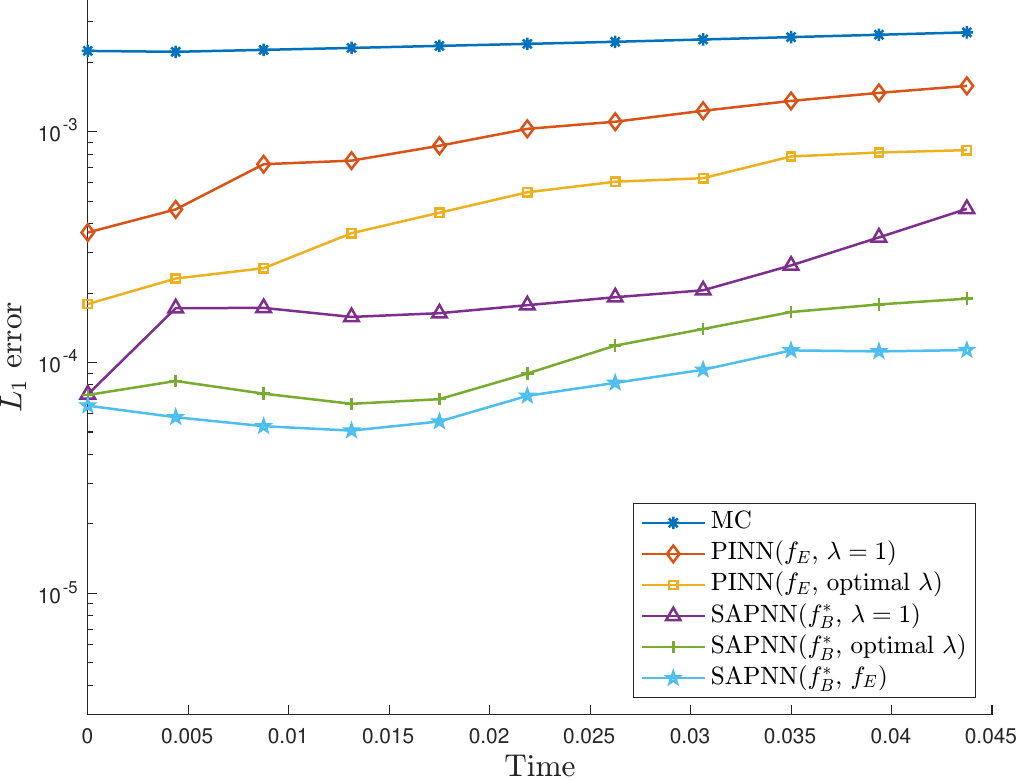}}}\\
        \caption{\sf $L_1$ error of expectation for Example \ref{exam4}. The number of samples used to compute the expected value and to construct the control variate are $K = 30$ and $L = 5000$, respectively. Left: $\eps = 10^{-2}$; Right: $\eps = 2 \times 10^{-4}$; Top: Density; Bottom: Energy.}
        \label{LaxError}
    \end{center}
\end{figure}

We note that this is a six-dimensional problem.
The initial condition consists of high pressure and high density on both the left and right sides, and low pressure and low density in the center, generating two rarefaction waves propagating in opposite directions and forming a vacuum or near-vacuum region in the middle.
The final time is set to $0.1$ with a relaxation parameter $\eps = 2 \times 10^{- 4}$.
The velocity space is truncated to the domain $[-8, 8] \times [-8, 8]$.

Figure~\ref{Exp2} shows the expected values of momentum and energy for both $f_B$ and $f_E$, demonstrating accurate performance.
In the left panel of Fig.~\ref{ErrorDR}, we observe that for $K = 30$ and $L = 2500$, both SAPNN($f_B$) and PINN($f_E$) significantly reduce the error compared to the MC method.
As anticipated, SAPNN($f_B$) achieves a more significant error reduction than PINN($f_E$), since the BGK model provides a closer approximation to the Boltzmann model than the Euler model.
In the right panel of Fig.\ref{ErrorDR}, increasing the sample size $L$ from $2500$ to $10000$ further improves the accuracy of SAPNN($f_B$).
However, for PINN($f_E$), increasing the sample size does not lead to further improvement, as both model error and generalization error become dominant.
Lastly, we emphasize that without a structure-preserving framework, PINN-based simulations become unstable in regions near vacuum.

\begin{figure}[htb]
    \begin{center}
        \mbox{
        {\includegraphics[width = 0.4 \textwidth, trim=0 0 0 0,clip]{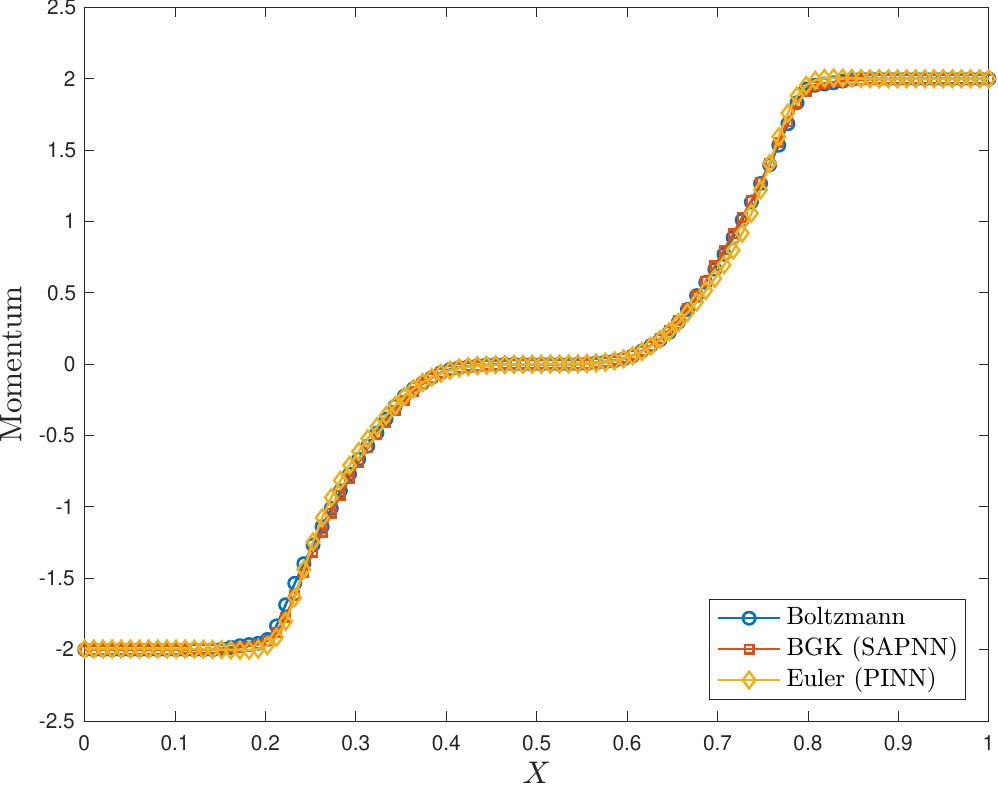}}
        {\includegraphics[width = 0.4 \textwidth, trim=0 0 0 0,clip]{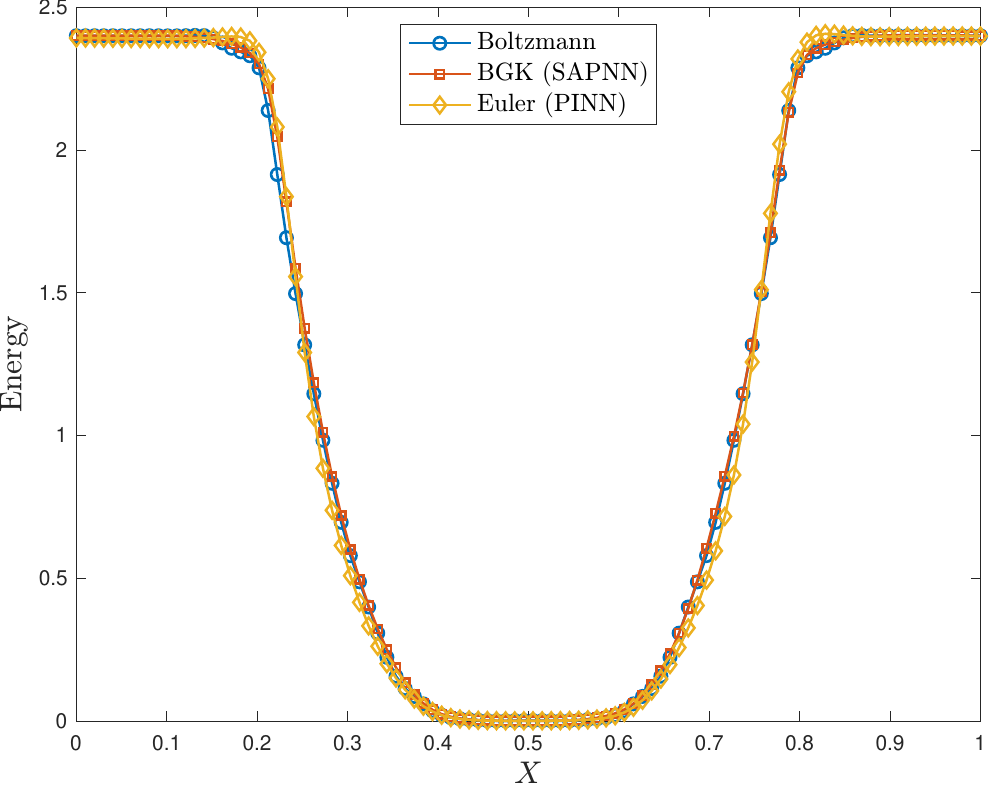}}}\\
        \caption{\sf Expectation for Example \ref{exam2}. Left: Momentum; Right: Energy.}
        \label{Exp2}
    \end{center}
\end{figure}

\begin{figure}[htb]
    \begin{center}
        \mbox{
        {\includegraphics[width = 0.4 \textwidth, trim=0 0 0 0,clip]{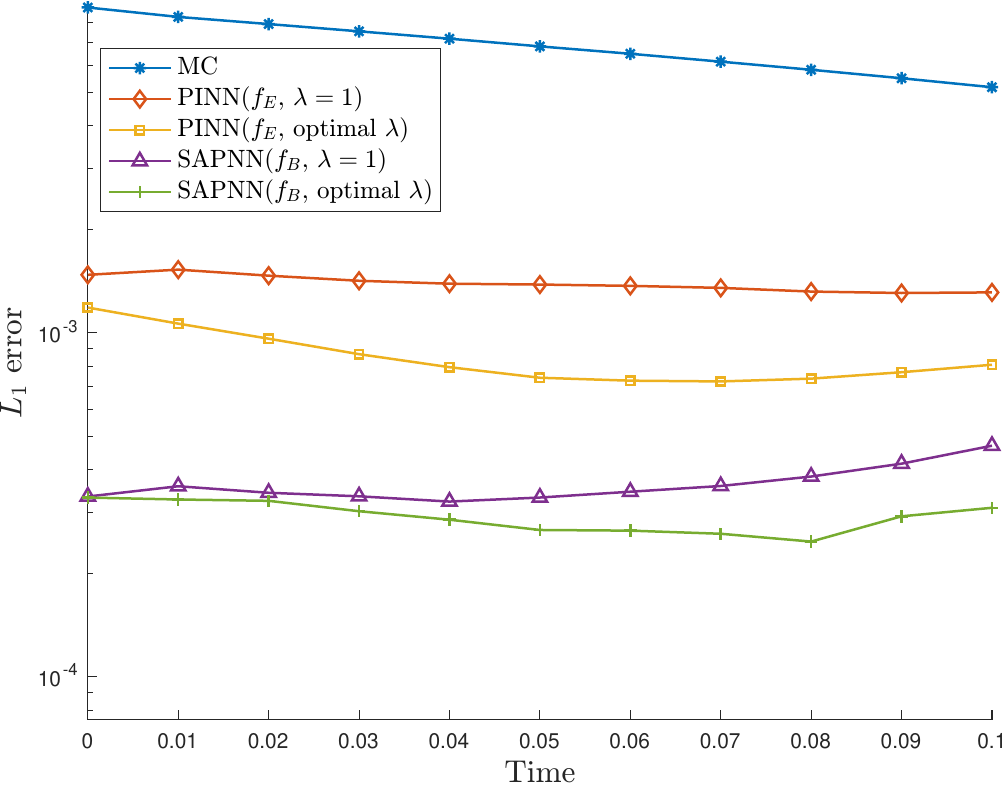}}
        {\includegraphics[width = 0.4 \textwidth, trim=0 0 0 0,clip]{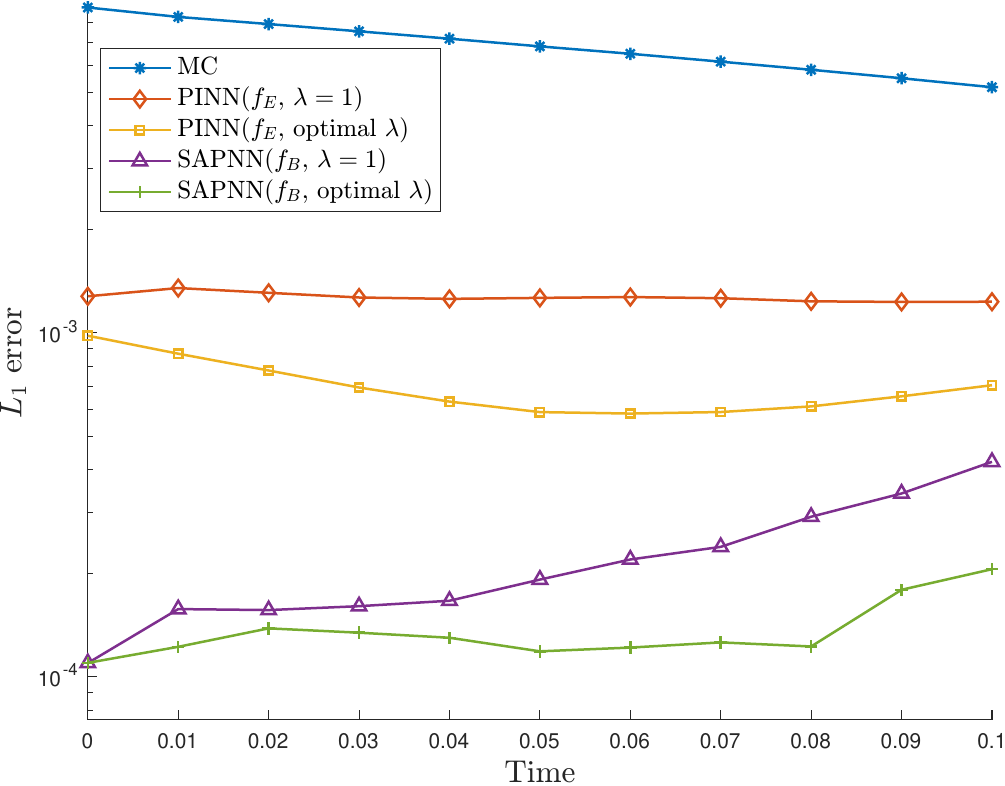}}}\\
        \mbox{
        {\includegraphics[width = 0.4 \textwidth, trim=0 0 0 -20,clip]{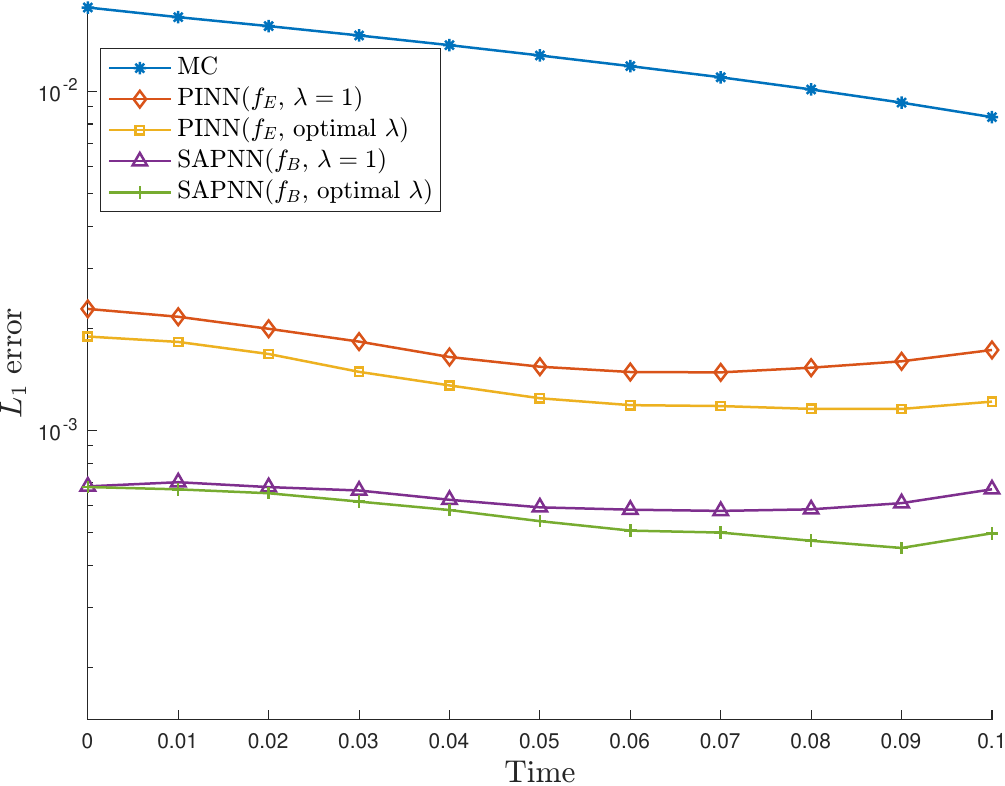}}
        {\includegraphics[width = 0.4 \textwidth, trim=0 0 0 -20,clip]{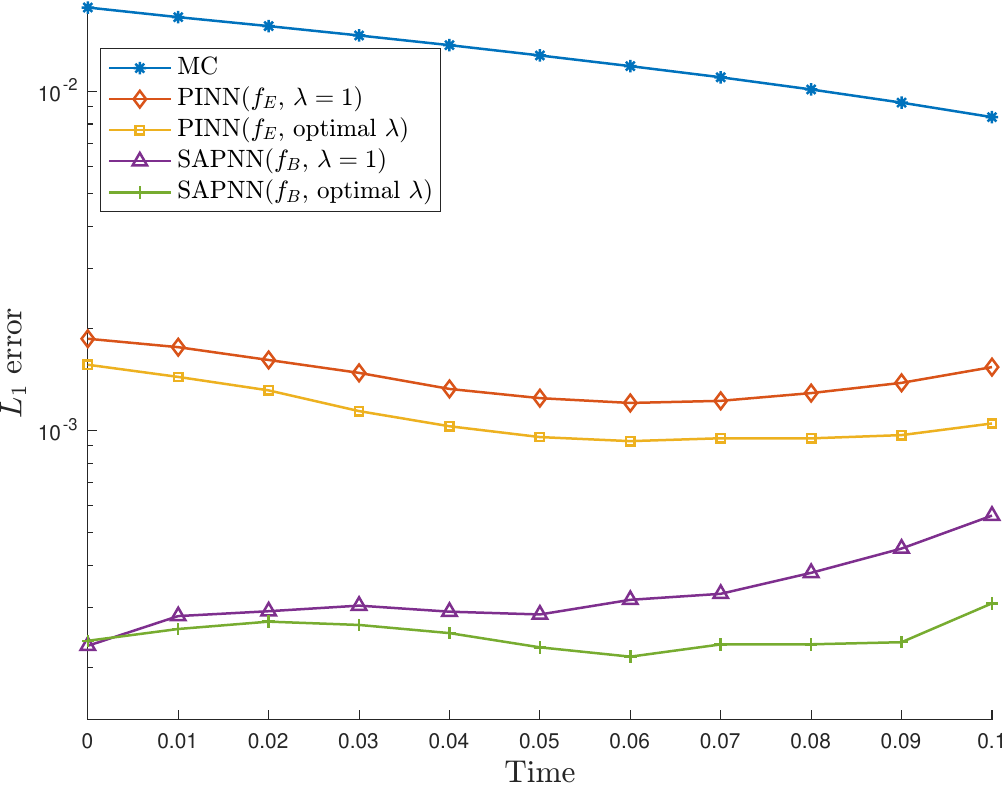}}}\\
        \caption{\sf $L_1$ errors of expectation for  Example \ref{exam2}. The number of samples used to compute the expected value is $K = 30$. Left: $L = 2500$; Right: $L = 10000$; Top: Momentum; Bottom: Energy.}
        \label{ErrorDR}
    \end{center}
\end{figure}

\section{Conclusion}
\label{sec6}
This work presents a neural network based multiscale control variate framework for uncertainty quantification in kinetic equations. By integrating structure and asymptotic preserving neural networks for BGK models and Euler equations enhanced with Boltzmann data with the control variate strategy, we achieve significant improvements in computational efficiency over traditional Monte Carlo sampling methods. The proposed SAPNN architecture ensures physical consistency and asymptotic-preservation even in multiscale settings, while its generalization to parametric uncertainties enables large-scale sample predictions. Numerical experiments for space homogeneous and non-homogeneous settings demonstrate the framework's capability to mitigate the curse of dimensionality and reduce variance effectively. Future research will focus on extending this methodology to other situations where evaluating the high fidelity kinetic model is significantly demanding when high-dimensional uncertainties are present, like diffusion limits \cite{Faou19} and collisional plasmas \cite{Medaglia23} characterized by the Landau operator. Another interesting research direction would amount in constructing SAPNN aimed at learning the whole action of the surrogate operators \cite{kovachki2021neural,FrankRN23}.

\bibliographystyle{siamplain}
\bibliography{references}
\end{document}